\documentclass[a4paper,12pt,twoside,reqno]{amsart}
\usepackage{mathptmx,mathrsfs,amssymb,amsmath,stmaryrd,cases}     
\usepackage{cite}
\usepackage[left=1in,right=1in,top=1in,bottom=1in]{geometry}
\usepackage{hyperref}

\newcommand{\norm}[1]{\Vert#1\Vert}
\newcommand{\lnorm}[1]{\left\Vert#1\right\Vert}
\newcommand{\abs}[1]{\left\vert#1\right\vert}

\newcommand{\R}{\mathbb R}
\newcommand{\D}{\partial}

\newcommand{\eps}{\varepsilon}

\newcommand{\dv}{\mathrm{div}\,}
\newcommand{\B}{\mathbf{B}}
\newcommand{\vv}{\mathbf{v}}
\newcommand{\vH}{\mathbf{H}}
\newcommand{\el}{\mathbf{E}}
\newcommand{\n}{\mathbf{n}}
\newcommand{\dm}{{\Omega_t^+}}
\newcommand{\idm}{{\mathscr{D}}}
\newcommand{\nb}{\nabla}
\newcommand{\bnb}{\overline{\nb}}
\newcommand{\tr}{\mathrm{tr}\,}
\newcommand{\curl}{\nb\times }
\newcommand{\dist}{\mathrm{dist}\,}
\newcommand{\ls}{\leqslant\,}
\newcommand{\gs}{\geqslant\,}
\newcommand{\N}{\mathbf{n}}

\newcommand{\vol}{\mathrm{Vol}\,}
\newcommand{\sgn}{\mathrm{sgn}}
\newcommand{\p}{ q }

\newcommand{\no}{\nonumber}
\newcommand{\K}{\mathcal{K}}
\newcommand{\E}{\mathcal{E}}
\newcommand{\T}{\mathcal{T}}
\newcommand{\bic}{\complement}
\newcommand{\jump}[1]{\left\llbracket#1\right\rrbracket}
\allowdisplaybreaks
\usepackage{amsthm}
\newtheorem{theorem}{Theorem}[section]

\newtheorem{lemma}[theorem]{Lemma}
\newtheorem{proposition}[theorem]{Proposition}
\theoremstyle{definition}
\newtheorem{definition}[theorem]{Definition}
\theoremstyle{remark}
\newtheorem{remark}[theorem]{Remark}

\numberwithin{equation}{section}

\begin{document}

\title[Motion of free ideal incompressible MHD]{On the motion of free interface in ideal incompressible MHD}%

\author[C.C.Hao]{Chengchun Hao}%

\address{Institute of Mathematics,
	Academy of Mathematics \& Systems Science,
	Chinese Academy of Sciences,  Beijing 100190, China
}
\email{hcc@amss.ac.cn}
\thanks{ The author would like to thank Professor T. Luo for his helpful discussion. Hao's research was 
	supported in part by the National Natural Science Foundation of China under grants 11171327 and 11671384.}

\maketitle

\begin{abstract}
	For the free boundary problem of the plasma-vacuum interface to three-dimensional ideal incompressible magnetohydrodynamics (MHD), the a priori estimates of smooth solutions are proved in Sobolev norms by adopting a geometrical point of view and some quantities such as the second fundamental form and the velocity of the free interface are estimated. In the vacuum region, the magnetic fields are described by the div-curl system of pre-Maxwell dynamics, while at the interface the total pressure is continuous and the magnetic fields are tangent to the interface, but we do not need any restrictions on the size of the magnetic fields on the free interface.  We introduce the ``virtual particle'' endowed with a virtual velocity field in vacuum to reformulate the problem to a fixed boundary problem under the Lagrangian coordinates. The $L^2$-norms of any order covariant derivatives of the magnetic fields both in vacuum and on the boundaries are bounded in terms of initial data and the second fundamental forms of the free interface and the rigid wall. The estimates of the curl of the electric fields in vacuum are also obtained, which are also indispensable in elliptic estimates.
	
\end{abstract}

\section{Introduction}

In the present paper, we are concerned with the free boundary problem of ideal incompressible magnetohydrodynamics (MHD). It consists of finding a bounded variable domain $\dm\subset \R^3$ filled with inviscid incompressible electrically conducting homogeneous plasma (the density is a positive constant), together with the vector field of velocity $\vv(t,x)=(v_1,v_2,v_3)$, the scalar pressure $p(t,x)$ and the magnetic field $\vH(t,x)=(H_1,H_2,H_3)$ satisfying the system of equations of MHD. The boundary $\Gamma_t$ of $\dm$ is the free surface of the plasma. It is assumed that the plasma is surrounded by a vacuum region $\Omega_t^-$ and that the whole domain $\Omega=\dm\cup \Gamma_t\cup \Omega_t^-$ is independent of time and bounded by a fixed perfectly conducting rigid wall $W$ such that $W\cap \Gamma_t=\emptyset$. Both $\dm$ and $\Omega$ are simply connected. The magnetic field should be found not only in $\dm$ but also in $\Omega_t^-$.

In the plasma region $\dm$, the ideal MHD equations apply, i.e., for $t>0$
\begin{subequations}\label{mhd}
	\begin{numcases}{}
	\vv_t+\vv\cdot\nb \vv+\nabla p=\mu\big(\vH\cdot\nb \vH-\frac{1}{2}\nb |\vH|^2\big),   \label{mhd.1}\\
	\vH_t+\vv\cdot \nb \vH=\vH\cdot \nb\vv,  \label{mhd.2}\\
	\dv \vv=0,  \quad \dv \vH=0. \label{mhd.3}
	\end{numcases}
\end{subequations}

Let $\hat{\vH}$ be the magnetic field in the vacuum $\Omega_t^-$. Since the vacuum has no density, velocity, electric current (i.e., $\hat{\el}_t=0$), except the magnetic field, we have the pre-Maxwell equations in vacuum
\begin{align}\label{v}
\nb\times\hat{\vH}=0,\; \dv\hat{\vH}=0,\; \hat{\vH}_t=-\nb\times\hat{\el},\; \dv\hat{\el}=0.
\end{align}
At the wall $W$, the tangential component of the electric field and the normal component of the magnetic field must vanish, i.e.,
\begin{align}\label{vb}
\N\times \hat{\el}=0,\quad \N\cdot\hat{\vH}=0, \quad \text{on } W,
\end{align}
where $\N$ is the inward drawn unit normal to the boundary of $\Omega_t^-$.

The plasma-vacuum interface is now free to move since the plasma is surrounded by vacuum. Hence, $\vv\cdot\N|_{\Gamma_t}$ is unknown and arbitrary where $\N$ is the unit normal to $\Gamma_t$ pointing from the plasma to the vacuum. Thus, we need some non-trivial jump conditions that must be satisfied to connect the fields across the interface.  These arise from the divergence $\vH$ equation, Faraday's law and the momentum equation. A convenient way to obtain the desired relations is to assume that the plasma surface $\Gamma_t$ is moving with a normal velocity 
\begin{align}
V_\N\N=(\vv\cdot\N)\N,
\end{align}
where  $V_\N$ is the velocity of evolution of $\Gamma_t$ in the direction $\N$. The jump conditions are straightforward to derive in a reference frame moving with the fluid surface. Once these conditions are obtained, all that is then required is to convert back to the laboratory frame using the corresponding Galilean transformation (cf. \cite{IdealMHD}). From Maxwell's equations, we know that, at the interface $\Gamma_t$, the magnetic field and the electric field must satisfy the conditions
\begin{align}\label{bc2}
\jump{\N\cdot\vH}=0,\text{ and } \jump{\N\times\el}=0,
\end{align}
where $\jump{f}\equiv \hat{f}-f$ is the jump in a quantity across the interface. We assume that the plasma is a perfect conductor, i.e., $\el+\vv\times\vH=0$. This implies that in the plasma $[\N\cdot\vH]_{\Gamma_t}$ and $[\N\times \el-(\N\cdot \vv)\vH]_{\Gamma_t}$ are both automatically zero.  
Therefore, \eqref{bc2} reduces to
\begin{align}\label{bc5}
\N\cdot\vH=\N\cdot\hat{\vH}=0, \quad \N\times \hat{\el}=(\vv\cdot\N)\hat{\vH},\quad \text{on } \Gamma_t.
\end{align}
The first one also means that the magnetic fields are not pointing into the vacuum on the interface.

We also have the following pressure balance condition (also cf. \cite{IdealMHD,SchMHD09}) on the interface $\Gamma_t$:
\begin{align}
&\jump{p+\frac{\mu}{2}|\vH|^2}=0,\quad \text{on } \Gamma_t.\label{bc1}
\end{align}

For convenience, we denote
$$P=q^+-q^-,\; q^+=p+\frac{\mu}{2}|\vH|^2,\;\text{and } q^-=\frac{\mu}{2}|\hat{\vH}|^2.$$
The system can be rewritten as 
\begin{subequations}\label{mhd1}
	\begin{numcases}{}
	\vv_t+\vv\cdot\nb \vv+\nb q^+=\mu\vH\cdot\nb \vH,  \quad\quad\quad\quad\quad\quad\quad\quad\;  \text{in }\dm, \label{mhd1.1}\\
	\vH_t+\vv\cdot \nb \vH=\vH\cdot \nb\vv, \quad\quad\quad\quad\quad\quad\quad\quad\quad\quad\quad\;\;  \text{in }\dm,\label{mhd1.2}\\
	\dv \vv=0,  \quad \dv \vH=0 , \quad\quad\quad\quad\quad\quad\quad\quad\quad\quad\quad\;   \text{in }\dm, \label{mhd1.3}\\
	\nb\times\hat{\vH}=0,\; \dv\hat{\vH}=0, \; \hat{\vH}_t=-\nb\times\hat{\el}, \;\dv\hat{\el}=0,\;\;\;\;  \text{in }\Omega_t^-, \label{magnetic}\\
	P=0, \quad \vH\cdot\N=\hat{\vH}\cdot\n=0,\;\quad\quad\quad\quad\quad\quad\quad\quad\quad\;\; \text{on }\Gamma_t,\label{bc.1}\\
	\N\times\hat{\el}=(\vv\cdot\N) \hat{\vH}, \quad\quad\quad\quad\quad\quad\quad\quad\quad\quad\quad\quad\quad\;\; \text{on }\Gamma_t,\label{bc.3}\\
	\N\times\hat{\el} =0, \quad \hat{\vH}\cdot\N=0, \quad\quad\quad\quad\quad\quad\quad\quad\quad\quad\quad\; \text{on } W,\label{bc.4}\\
	\vv(0,x)=\vv_0(x),\; \vH(0,x)=\vH_0(x),\;\quad\quad\quad\quad\quad\quad\;\;\;\; \text{in }\Omega^+,\label{initialcondition}\\
	\hat{\vH}(0,x)=\hat{\vH}_0(x),\;\hat{\el}(0,x)=\hat{\el}_0(x),\;\quad\quad\quad\quad\quad\quad\;\;\; \text{in }\Omega^-, \label{initialcondition'}\\
	\Omega_t^+|_{t=0}=\Omega^+,\quad \Omega_t^-|_{t=0}=\Omega^-, \quad \Gamma_t|_{t=0}=\Gamma,
	\end{numcases}
\end{subequations}
where $\hat{\el}_0$ satisfies the boundary condition \eqref{bc.3} and \eqref{bc.4}, i.e., $N\times\hat{\el}_0=(\vv_0\cdot N)\hat{\vH}_0$  on $\Gamma$ and $\N\times\hat{\el}_0=0$ on $W$, where $N$ denotes the unit normal to $\Gamma$ pointing from the plasma to the vacuum.

We will prove a priori bounds for the  interface problem \eqref{mhd1} in Sobolev spaces under the following condition
\begin{align}\label{phycond}
\nb_N P\ls -\eps<0 \text{ on } \Gamma,
\end{align}
where $\nb_N=N^i\D_i$ indicates the normal derivative. That is, we  assume that this condition holds initially, and then we can verify that it holds true within a period. For the free boundary problem of motion of incompressible fluids in vacuum, without magnetic fields, the natural physical condition (cf. \cite{bealhou,CL,Coutand,L1,L2,LN,Ebin,SZ,wu1, wu2,zhang}) reads that
$ \nb_{N} p\ls -\eps<0 \text{ on } \Gamma$, which excludes the possibility of the Rayleigh-Taylor type instability (see \cite{Ebin}).

Up to now, there were no well-posedness results for full non-stationary plasma-vacuum models. This is due to the difficulties caused by the strong coupling between the velocity field and the magnetic field. Hao and Luo studied the a priori estimates for the free boundary problem of ideal incompressible MHD flows in \cite{HLarma} with a bounded initial domain homeomorphic to a ball, provided that the size of the magnetic field to be invariant on the free boundary. 
For the special case when the magnetic field is zero on the free boundary, Lee proved the local existence and uniqueness of plasma-vacuum free boundary problem of incompressible viscous-diffusive MHD flow in three-dimensional space with infinite depth setting  in \cite{Lee14a}, and got a local unique solution of free boundary MHD without kinetic viscosity and magnetic diffusivity via zero kinetic viscosity-magnetic diffusivity limit in \cite{Lee14b}. 
For the incompressible viscous MHD equations, a free boundary problem  in a simply connected domain of $\R^3$ was studied by a linearization technique and the construction of a sequence of successive approximations in \cite{PS10} with an irrotational condition for magnetic fields in a part of the domain. 
The well-posedness of the linearized plasma-vacuum interface problem in ideal incompressible MHD was studied in \cite{MTT14} in an unbounded plasma domain, the linearized plasma-vacuum problem in ideal compressible MHD was investigated in \cite{tr1,ST13},  the well-posedness of the original nonlinear free boundary problem was proved in \cite{ST14} by using the Nash-Moser method, and a stationary problem was studied in \cite{FL95}. In \cite{CMST12}, the a priori estimates for smooth solutions of the free boundary problem for current-vortex sheets in ideal incompressible two-fluid MHD  was proved in the domain $\mathbb{T}^2\times(-1,1)$ under some linearized stability conditions on the jump function of the velocity field and the magnetic fields. The existence of current-vortex sheets problem in compressible MHD was studied in \cite{chenwang} and \cite{tr2}. 

Regarding the cases without magnetic fields, the free surface problem of the incompressible Euler equations of fluids has attracted much attention in recent decades and  important progress has been made for flows with or without vorticity or surface tension. We refer readers to \cite{AM,CL,Coutand,GMS12,L1,L2,LN,Ebin,SZ,wu1,wu2,Wu11,zhang} and references therein.

In this paper, we prove a priori bounds for the free interface problem \eqref{mhd1} and take into account the  magnetic field not only in plasma but also in vacuum. We do not need the restricted boundary condition $ |\vH|\equiv \text{const}\gs 0$ on the free interface assumed in \cite{HLarma}. What makes this problem difficult is that the regularity of the boundary enters to the highest order and  energies interchange  between plasma and vacuum. Roughly speaking, the energies of plasma and vacuum will exchange via the pressure balance relation $q^+=q^-$ on the free interface. But we have to investigate the estimates of the magnetic field and the electric field in vacuum although the electric field is only a secondary variable in order to obtain energy estimates. We also introduce the ``virtual particle'' endowed with a virtual velocity field to reformulate the free boundary problem to a fixed boundary problem. We can show that the norms of the magnetic field in vacuum depend only on the norms of initial data and the second fundamental forms of the interface and the wall.

Now, we derive the conserved energy. Let $D_t:=\D_t+v^k\D_k$ be the material derivative, then it holds for any function $F$ on $\overline{\dm}$
\begin{align} \label{dtint}
\frac{d}{dt}\int_{\dm} Fdx=\int_\dm D_tF dx,
\end{align}
since $\dv \vv=0$, and then for any function $F$ on $\overline{\Omega_t^-}$ (we can extend it to $\Omega$ by a smooth cut-off function)
\begin{align}
\frac{d}{dt}\int_{\Omega_t^-} Fdx=&\frac{d}{dt}\int_{\Omega} Fdx-\frac{d}{dt}\int_{\dm} Fdx=\int_{\Omega} \D_tFdx-\int_\dm D_tF dx\no\\
=&\int_{\Omega_t^-}\D_t Fdx-\int_{\dm}\vv\cdot\nb Fdx
=\int_{\Omega_t^-}\D_t Fdx-\int_{\Gamma_t} \vv_\N FdS,\label{dtintv}
\end{align}
where $\N$ is the outward unit normal to $\Gamma_t$ corresponding to $\Omega_t^+$.

Throughout the paper, we use the Einstein summation convention, that is, when an index variable appears twice in both the subscript and the superscript of a single term it indicates summation of that term over all the values of the index.

From \eqref{dtint} and \eqref{dtintv}, we have, by using the boundary conditions \eqref{bc.1}, \eqref{bc.3} and \eqref{bc.4}, that
\begin{align*}
&\frac{d}{dt}\left(\int_\dm \left(\frac{1}{2}|\vv|^2+\frac{\mu}{2}|\vH|^2\right) dx+\frac{\mu}{2}\int_{\Omega_t^-}|\hat{\vH}|^2 dx\right)\\
=&\int_\dm \left(v^i D_t v_i+\mu H^iD_t H_i\right) dx
+\frac{\mu}{2}\int_{\Omega_t^-}\D_t|\hat{\vH}|^2dx-\frac{\mu}{2}\int_{\Gamma_t} \vv_\N|\hat{\vH}|^2 dS\\
=&\int_\dm \left[v^i(-\D_ip+\mu H^k\D_kH_i-\frac{\mu}{2}\D_i|\vH|^2)+\mu H^iH^k\D_kv_i\right]dx\\
&+\mu \int_{\Omega_t^-}\dv(\hat{\vH}\times\hat{\el})dx-\frac{\mu}{2}\int_{\Gamma_t} \vv_\N|\hat{\vH}|^2 dS\\
=&-\int_{\Gamma_t}\left(p+\frac{\mu}{2}|\vH|^2\right)\vv_\N dS +\mu \int_{\Gamma_t} (\vH\cdot\N)(\vv\cdot\B)dS\\
&+\mu \int_{\Gamma_t}\hat{\vH}\cdot(\N\times\hat{\el})dS-\frac{\mu}{2}\int_{\Gamma_t} \vv_\N|\hat{\vH}|^2 dS+\mu \int_{W}\hat{\vH}\cdot(\hat{\el}\times \N)dS\\
=&-\int_{\Gamma_t}\left(p+\frac{\mu}{2}|\vH|^2-\frac{\mu}{2}|\hat{\vH}|^2\right)\vv_\N dS\\
=&0,
\end{align*}
due to
\begin{align*}
\frac{1}{2}\D_t|\hat{\vH}|^2=-\hat{\vH}\cdot (\nb\times\hat{\el})
=\dv(\hat{\vH}\times\hat{\el})-\hat{\el}\cdot(\curl\hat{\vH})=\dv(\hat{\vH}\times\hat{\el}),
\end{align*}
in view of $\curl\hat{\vH}=0$ in $\Omega_t^-$.
Thus, the conserved physical energy can be given by
\begin{align*}
E_0(t):=&\int_\dm \left(\frac{1}{2}|\vv(t)|^2+\frac{\mu}{2}|\vH(t)|^2\right) dx+\int_{\Omega_t^-}\frac{\mu}{2}|\hat{\vH}(t)|^2 dx\equiv E_0(0),
\end{align*}
for any $t>0$.

The higher order energy has an interface boundary part and interior parts in both plasma and vacuum. The boundary part controls the norms of the second fundamental form of the free interface, the interior part in plasma controls the norms of the velocity, magnetic fields and hence the pressure, the interior part in vacuum controls the magnetic field. We will prove that the time derivatives of the energy are controlled by themselves. A crucial point in the construction of the higher order energy norms is that the time derivatives of the interior parts will, after integrating by parts, contribute some boundary terms that cancel
the leading-order  terms in the corresponding time derivatives of the boundary integrals. To this end,  we need to project the equations  for the total pressure $P$ to the tangent space of the boundary.

The orthogonal projection $\Pi$ to the tangent space of the boundary of a $(0, r)$ tensor $\alpha$ is defined to be the projection of each component along the normal:
\begin{align*}
(\Pi \alpha)_{i_1\cdots i_r}=\Pi_{i_1}^{j_1}\cdots \Pi_{i_r}^{j_r} \alpha_{j_1\cdots j_r},  \quad \text{ where } \Pi_i^j=\delta_i^j-{\N}_i{\N}^j, \end{align*}
with ${\N}^j=\delta^{ij} {\N}_i={\N}_j$.

Let $\bar \partial_i=\Pi_i^j \partial_j$ be a tangential derivative. If $q=const$ on $\partial \dm$, it follows that $\bar \partial_i q=0$ there and
\begin{equation*}
(\Pi \partial^2 q)_{ij}=\theta_{ij} \nabla_{\N} q,
\end{equation*}
where ${\theta}_{ij}=\bar \partial_i \N_j$ is the second fundamental form of $\partial \dm$.

The higher order energies are defined as:
For $r\gs 1$
\begin{align*}
E_r(t)=&\int_{\dm}\delta^{ij}Q(\D^r v_i, \D^r v_j)dx+\mu \int_{\dm}\delta^{ij}Q(\D^r H_i, \D^r H_j)dx\no\\
&+\mu \int_{\Omega_t^-} \delta^{ij}Q(\D^r\hat{H}_i,\D^r\hat{H}_j)dx+\sgn(r-1)\int_{\Gamma_t}Q(\D^r P, \D^r P)\vartheta dS\no\\
&+\int_{\dm}\left(|\D^{r-1}\curl v|^2+\mu |\D^{r-1}\curl B|^2\right)dx,
\end{align*}
where  $\sgn(s)$ be the sign function of the real number $s$ (so we do not need the boundary integral for $r=1$) and $\vartheta=(-\nabla_{\N} P)^{-1}$.
Here  $Q$ is a positive definite quadratic form which,
when restricted to the boundary, is the inner product of the tangential components, i.e., $Q(\alpha, \beta)=\langle\Pi \alpha, \Pi \beta\rangle$,  and in the interior $Q(\alpha, \alpha)$ increases to the norm $|\alpha|^2$. To be more specific, let
\begin{equation*}
Q(\alpha, \beta)=q^{i_1j_1}\cdots q^{i_rj_r}\alpha_{i_1\cdots i_r}\beta_{j_1\cdots j_r}
\end{equation*}
where
\begin{equation*}
q^{ij}=\delta^{ij}-\eta(d)^2\N^i\N^j, \quad d(x)=\dist(x, \Gamma_t), \quad \N^i=-\delta^{ij}\D_j d. 
\end{equation*}
Here  $\eta$ is a smooth cutoff function satisfying $0\ls \eta(d)\ls 1, \quad \eta(d)= 1$ when
$d < d_0/4$ and $\eta(d)=0$ when $d > d_0/2. $ $d_0$  is a fixed number that is smaller than
the injectivity radius of the normal exponential map $\iota_0$, defined to be the largest
number $\iota_0$ such that the map
\begin{equation}\label{defn.3.4}
\Gamma_t\times (-\iota_0, \iota_0)\to \{x\in \R^3:  \dist(x,\  \Gamma_t) < \iota_0\},
\end{equation}
given by
$$(\bar x, \iota)\to x=\bar x+\iota \N(\bar x),$$
is an injection. We can also define $\iota_0$ for $W$ similarly which is independent of $t$ and denote it by $\iota_0'$. We always assume the second fundamental form of $W$ and $1/\iota_0'$ to be bounded, since both of them are invariant with respect to time.

Now, we can state the main theorem as follows.
\begin{theorem}\label{main thm}
	Let \begin{align}
	\K(0)=&\max\left(\norm{\theta(0,\cdot)}_{L^\infty(\Gamma)}, 1/\iota_0(0)\right),\label{K'}\\
	\E(0)=&\norm{1/(\nb_N P(0,\cdot))}_{L^\infty(\Gamma)}=1/\eps>0. \label{E'}
	\end{align}
	Assume $\norm{\theta}_{L^\infty(W)}+1/\iota_0'\ls C\K(0)$ and $\nb\times \hat{\el}_0\in L^2(\Omega^-)$. Then there exists a continuous function $\T>0$ such that if
	\begin{align}
	T\ls \T(\K(0),\E(0),E_0(0),\cdots, E_{4}(0),\vol\Omega ),
	\end{align}
	then any smooth solution of the free boundary problem for MHD equations \eqref{mhd1} with \eqref{phycond} in $[0,T]$ satisfies the estimate
	\begin{align}
	\sum_{s=0}^{4} E_s(t)\ls 2\sum_{s=0}^{4} E_s(0), \quad 0\ls t\ls T.
	\end{align}
\end{theorem}

The rest of this paper is organized as follows: In Section \ref{sec.2}, we use the Lagrangian coordinates to transform the free interface problem to a fixed initial boundary problem. Sections \ref{sec.3} and \ref{sec.4} are devoted to the estimates of the magnetic field and the electric field in vacuum, respectively. In Section \ref{sec.5}, we prove the higher order energy estimates. In the derivation of the higher order energy estimates in Section \ref{sec.5}, some a priori assumptions are made, which will be justified in section \ref{sec.6}. In order to make this paper more readable, we give an appendix on some estimates used in the previous sections, which are basically proved in \cite{CL}.

\section{Reformulation in Lagrangian Coordinates}\label{sec.2}

We may think that the velocity field of the ``virtual particles'' in vacuum  is $\vv$ on the boundary. Then, we can extend the velocity from the boundary to the interior of vacuum by a cut-off function such that
\begin{align*}
\vv(t,x)=\left\{
\begin{array}{ll}
\vv(t,\bar{x}), &\text{ near } \Gamma_t,\\
\text{smooth},& \text{ otherwise},\\
0,& \text{ near } W,
\end{array}
\right.
\end{align*}
and $\dv \vv=0$ for $x\in \Omega_t^-$ as long as $\Gamma_t\cap W=\emptyset$ in $[0,T]$, where $\bar{x}\in\Gamma_t$ satisfies $\dist(x,\Gamma_t)=|x-\bar{x}|$, and $|\vv(t,x)|\ls \norm{\vv(t)}_{L^\infty(\Gamma_t)}$ for $x\in \Omega^-_t$ by construction.

Assume that we are given a velocity vector field $\vv(t,x)$ defined in a set $\idm\subset [0,T]\times\R^n$ such that the interface of $\dm=\{x: (t,x)\in\idm\}$ and $\Omega_t^-$ moves with the velocity, i.e., $(1,\vv)\in T(\D\idm)$ which denotes the tangent space of $\D\idm$. We will now introduce Lagrangian or co-moving coordinates, that is, coordinates that are constant along the integral curves of the velocity vector field so that the boundary becomes fixed in these coordinates (cf. \cite{CL}). Let $x=x(t,y)=f_t(y)$ be the trajectory of the particles given by
\begin{align}\label{trajectory}
\left\{\begin{aligned}
&\frac{dx}{dt}=\vv(t,x(t,y)), \quad (t,y)\in [0,T]\times \Omega,\\
&x(0,y)=f_0(y), \quad y\in\Omega.
\end{aligned}\right.
\end{align}
where, when $t=0$, we can start with either the Euclidean coordinates in $\Omega$ or some other coordinates $f_0:\Omega\to \Omega$ where $f_0$ is a diffeomorphism in which the domain $\Omega$ becomes simple. For simplicity, we will assume $f_0(y)=y$ in this paper. Then, the Jacobian determinant $\det(\D y/\D x)\equiv 1$ due to the divergence-free property of $\vv$.  For each $t$, we will then have a change of coordinates $f_t:\Omega\to \Omega$, taking $y\to x(t,y)$.  The Euclidean metric $\delta_{ij}$ in $\Omega$ then induces a metric
\begin{align}\label{metric1}
g_{ab}(t,y)=\delta_{ij}\frac{\D x^i}{\D y^a}\frac{\D x^j}{\D y^b}
\end{align}
in $\Omega$ for each fixed $t$.

We use the covariant differentiation in $\Omega$ with respect to the metric $g_{ab}(t,y)$, because it corresponds to differentiation in $\Omega$ under the change of coordinates $\Omega \ni y\to x(t,y)\in\Omega$, and we will work in both coordinate systems. This also avoids possible singularities in the change of coordinates. We denote covariant differentiation in the $y_a$-coordinate by $\nb_a$, $a=0, \cdots, 3$, and differentiation in the $x_i$-coordinate by $\D_i$, $i=1,2,3$. The covariant differentiation of a $(0,r)$ tensor $k(t,y)$ is the $(0,r+1)$ tensor given by
\begin{align*}
\nb_a k_{a_1\cdots a_r}=\frac{\D k_{a_1\cdots a_r}}{\D y^a}-\Gamma_{aa_1}^d k_{d\cdots a_r}-\cdots-\Gamma_{aa_r}^dk_{a_1\cdots d},
\end{align*}
where the Christoffel symbols $\Gamma_{ab}^c$ are given by
\begin{align*}
\Gamma_{ab}^c=\frac{g^{cd}}{2}\left(\frac{\D g_{bd}}{\D y^a}+\frac{\D g_{ad}}{\D y^b}-\frac{\D g_{ab}}{\D y^d}\right)=\frac{\D y^c}{\D x^i}\frac{\D^2 x^i}{\D y^a \D y^b},
\end{align*}
where $g^{cd}$ is the inverse of $g_{ab}$. If $w(t,x)$ is the $(0,r)$ tensor expressed in the $x$-coordinates, then the same tensor $k(t,y)$ expressed in the $y$-coordinates is given by
\begin{align*}
k_{a_1\cdots a_r}(t,y)=\frac{\D x^{i_1}}{\D y^{a_1}}\cdots \frac{\D x^{i_r}}{\D y^{a_r}}w_{i_1\cdots i_r}(t,x), \quad x=x(t,y),
\end{align*}
and by the transformation properties for tensors,
\begin{align}\label{covtensor}
\nb_a k_{a_1\cdots a_r}=\frac{\D x^i}{\D y^a}\frac{\D x^{i_1}}{\D y^{a_1}}\cdots \frac{\D x^{i_r}}{\D y^{a_r}}\frac{\D w_{i_1\cdots i_r}}{\D x^i}.
\end{align}
Covariant differentiation is constructed so that the norms of tensors are invariant under changes of coordinates,
\begin{align}\label{norminv}
g^{a_1b_1}\cdots g^{a_rb_r} k_{a_1\cdots a_r}k_{b_1\cdots b_r}=\delta^{i_1j_1}\cdots \delta^{i_rj_r}w_{i_1\cdots i_r}w_{j_1\cdots j_r}.
\end{align}

Furthermore, we express in the $y$-coordinates,
\begin{align}\label{Di}
\D_i=\frac{\D}{\D x^i}=\frac{\D y^a}{\D x^i}\frac{\D}{\D y^a}.
\end{align}
Since the curvature vanishes in the $x$-coordinates, it must do so in the $y$-coordinates, and hence
\begin{align*}
[\nb_a,\nb_b]=0.
\end{align*}
Let us introduce the notation ${{k_{a\cdots}}^b}_{\cdots c}=g^{bd}k_{a\cdots d\cdots c}$, and recall that the covariant differentiation commutes with lowering and rising indices: $g^{ce}\nb_a k_{b\cdot e\cdots d}=\nb_a g^{ce}k_{b\cdot e\cdots d}$. We also introduce a notation for the material derivative
\begin{align*}
D_t=\left.\frac{\D}{\D t}\right|_{y=\textrm{const}}=\left.\frac{\D}{\D t}\right|_{x=\textrm{const}}+v^k\frac{\D}{\D x^k}.
\end{align*}
Then we have, from \cite[Lemma 2.2]{CL}, that
\begin{align}\label{Dt}
D_tk_{a_1\cdots a_r}=\frac{\D x^{i_1}}{\D y^{a_1}}\cdots \frac{\D x^{i_r}}{\D y^{a_r}}\left(D_t w_{i_1\cdots i_r}+\frac{\D v^\ell}{\D x^{i_1}}w_{\ell\cdots i_r}+\cdots +\frac{\D v^\ell}{\D x^{i_r}}w_{i_1\cdots\ell}\right).
\end{align}

We recall a result concerning time derivatives of the change of
coordinates and commutators between time derivatives and space derivatives (cf. \cite[Lemma 2.1]{CL,HLarma}).

\begin{lemma}\label{lem.CL2.1}
	Let $x=f_t(y)$ be the change of variables given by \eqref{trajectory}, and let $g_{ab}$ be the metric given by \eqref{metric1}. Let $v_i=\delta_{ij}v^j=v^i$, and set
	\begin{align}
	u_a(t,y)=&v_i(t,x)\frac{\D x^i}{\D y^a},  & u^a=&g^{ab}u_b,\label{CL2.1.1}\\
	h_{ab}=&\frac{1}{2}D_t g_{ab}, & h^{ab}=&g^{ac}h_{cd}g^{db}.\label{CL2.1.2}
	\end{align}
	Then
	\begin{align}\label{CL2.1.3}
	D_t\frac{\D x^i}{\D y^a}=\frac{\D x^k}{\D y^a} \frac{\D v^i}{\D x^k},\quad D_t\frac{\D y^a}{\D x^i}=-\frac{\D y^a}{\D x^k}\frac{\D v^k}{\D x^i},
	\end{align}
	\begin{align}\label{CL2.1.4}
	&D_t g_{ab}=\nb_a u_b+\nb_b u_a, \quad D_t g^{ab}=-2h^{ab},  \quad D_t d\mu_g=\tr h d\mu_g,
	\end{align}
	where $d\mu_g$ is the Riemannian volume element on $\Omega$ in the metric $g$.
\end{lemma}

We now recall the estimates of commutators between the material derivative $D_t$ and space derivatives $\D_i$ and covariant derivatives $\nb_a$.

\begin{lemma}[cf. \cite{CL}]
	Let $\D_i$ be given by \eqref{Di}. Then
	\begin{align}\label{DtDi}
	[D_t,\D_i]=-(\D_i v^k)\D_k.
	\end{align}
	Furthermore,
	\begin{align}\label{DtDir}
	[D_t,\D^r]=-\sum_{s=0}^{r-1}\bic_{r}^{s+1}(\D^{1+s}v)\cdot \D^{r-s},
	\end{align}
	where $\bic_r^s$ denotes the binomial coefficient defined by $\frac{r!}{(r-s)!s!}$ for $0\ls s\ls r$, the symmetric dot product is defined to be in components
	\begin{align*}
	\left((\D^{1+s}v)\cdot \D^{r-s}\right)_{i_1\cdots i_r}=\frac{1}{r!}\sum_{\sigma\in \Sigma_r}\left(\D_{i_{\sigma_1}\cdots i_{\sigma_{1+s}}}^{1+s} v^k\right)\D_{ki_{\sigma_{s+2}}\cdots i_{\sigma_r}}^{r-s},
	\end{align*}
	and $\Sigma_r$ denotes the collection of all permutations of $\{1,2,\cdots, r\}$.
\end{lemma}

\begin{lemma}[cf. \cite{CL,HLarma}] \label{lem.CL2.4}
	Let $T_{a_1\cdots a_r}$ be a $(0,r)$ tensor. We have
	\begin{align}\label{CL2.4.1}
	[D_t,\nb_a]T_{a_1\cdots a_r}=-(\nb_{a_1}\nb_a u^d)T_{da_2\cdots a_r}-\cdots -(\nb_{a_r}\nb_a u^d)T_{a_1\cdots a_{r-1}d}.
	\end{align}
	If $\Delta=g^{cd}\nb_c\nb_d$ and $q$ is a function, we have
	\begin{align}
	[D_t,g^{ab}\nb_a]T_b=&-2h^{ab}\nb_a T_b-(\Delta u^e)T_e,\label{CL2.4.2}\\
	[D_t,\nb]q=&0, \label{Dtpcommu}\\
	[D_t,\Delta]q=&-2h^{ab}\nb_a\nb_b q-(\Delta u^e)\nb_e q.\label{CL2.4.3}
	\end{align}
	Furthermore, for $r\gs 2$,
	\begin{align}\label{commutator}
	[D_t,\nb^r]q=\sum_{s=1}^{r-1}-\bic_{r}^{s+1}(\nb^{s+1}u)\cdot \nb^{r-s} q,
	\end{align}
	where the symmetric dot product is defined to be in components
	\begin{align*}
	\left((\nb^{s+1}u)\cdot \nb^{r-s} q\right)_{a_1\cdots a_r}=\frac{1}{r!}\sum_{\sigma\in\Sigma_r}\left(\nb_{a_{\sigma_1}\cdots a_{\sigma_{s+1}}}^{s+1} u^d\right)\nb_{da_{\sigma_{s+2}}\cdots a_{\sigma_r}}^{r-s}q.
	\end{align*}
\end{lemma}

\begin{remark}
	It follows from \eqref{commutator} that for $r\gs 2$ and a function $q$,
	\begin{align*}
	D_t\nb^r q+\nb^ru\cdot\nb q=\nb^rD_tq-\sgn(r-2)\sum_{s=1}^{r-2}\bic_{r}^{s+1}(\nb^{s+1}u)
	\cdot \nb^{r-s} q.
	\end{align*}
\end{remark}

Denote
\begin{align*}
& H_i=\delta_{ij} H^j= H^i, \quad \beta_a= H_j\frac{\D x^j}{\D y^a},\quad \beta^a=g^{ab}\beta_b, \quad
|\beta|^2=\beta_a\beta^a,\\
&\hat{H}_i=\delta_{ij}\hat{H}^j=\hat{H}^i, \quad \varpi_a=\hat{H}_j\frac{\D x^j}{\D y^a}, \quad \varpi^a=g^{ab}\varpi_b,\quad |\varpi|^2=\varpi^a\varpi_a,\\
&\hat{E}_i=\delta_{ij}\hat{E}^j=\hat{E}^i, \quad \varXi_a=\hat{E}_j\frac{\D x^j}{\D y^a}, \quad \varXi^a=g^{ab}\varXi_b,\quad |\varXi|^2=\varXi^a\varXi_a.
\end{align*}
It follows from \eqref{norminv} that
\begin{align}\label{beta}
|\beta|=|\vH|, \; |\varpi|=|\hat{\vH}|, \; |\varXi|=|\hat{\el}|,\;  H_j=\frac{\D y^a}{\D x^j}\beta_a, \; \hat{H}_j=\frac{\D y^a}{\D x^j}\varpi_a, \; \hat{E}_j=\frac{\D y^a}{\D x^j}\varXi_a.
\end{align}

From \eqref{CL2.1.3} and \eqref{covtensor}, we have
\begin{align*}
D_t \varpi_a=&D_t \left(\hat{H}_j\frac{\D x^j}{\D y^a}\right) =\frac{\D x^j}{\D y^a}D_t  \hat{H}_j+\hat{H}_jD_t \frac{\D x^j}{\D y^a}\\
=&\frac{\D x^j}{\D y^a}\left(-(\curl\hat{\el})_j+v^k\D_k\hat{H}_j\right)+\hat{H}_j \frac{\D x^k}{\D y^a} \frac{\D v^j}{\D x^k}\\
=&-\frac{\D x^j}{\D y^a}(\curl\hat{\el})_j+\frac{\D x^j}{\D y^a}\frac{\D x^k}{\D y^b}\frac{\D y^b}{\D x^l}v^l\D_k\hat{H}_j +\hat{H}_j \frac{\D x^k}{\D y^a} \frac{\D v^l}{\D x^k}\frac{\D x^j}{\D y^b}\frac{\D y^b}{\D x^l}\\
=&-\frac{\D x^j}{\D y^a}(\curl\hat{\el})_j+u^b\nb_b\varpi_a+\varpi_b\nb_a u^b.
\end{align*}
Due to $\det(\D y/\D x)\equiv 1$, we get
\begin{align*}
\frac{\D x^j}{\D y^a}(\curl\hat{\el})_j=&\frac{\D x^1}{\D y^a}\left(\frac{\D\hat{E}_3}{\D x^2}-\frac{\D\hat{E}_2}{\D x^3}\right)+\frac{\D x^2}{\D y^a}\left(\frac{\D\hat{E}_1}{\D x^3}-\frac{\D\hat{E}_3}{\D x^1}\right)+\frac{\D x^3}{\D y^a}\left(\frac{\D\hat{E}_2}{\D x^1}-\frac{\D\hat{E}_1}{\D x^2}\right)\\
=&\frac{\D x^1}{\D y^a}\frac{\D x^1}{\D y^d}\frac{\D y^d}{\D x^1}\left(\frac{\D x^2}{\D y^b}\frac{\D y^b}{\D x^2}\frac{\D x^3}{\D y^c}\frac{\D y^c}{\D x^3}\frac{\D\hat{E}_3}{\D x^2}-\frac{\D x^3}{\D y^b}\frac{\D y^b}{\D x^3}\frac{\D x^2}{\D y^c}\frac{\D y^c}{\D x^2}\frac{\D\hat{E}_2}{\D x^3}\right)\\
&+\frac{\D x^2}{\D y^a}\frac{\D x^2}{\D y^d}\frac{\D y^d}{\D x^2}\left(\frac{\D x^3}{\D y^b}\frac{\D y^b}{\D x^3}\frac{\D x^1}{\D y^c}\frac{\D y^c}{\D x^1}\frac{\D\hat{E}_1}{\D x^3}-\frac{\D x^1}{\D y^b}\frac{\D y^b}{\D x^1}\frac{\D x^3}{\D y^c}\frac{\D y^c}{\D x^3}\frac{\D\hat{E}_3}{\D x^1}\right)\\
&+\frac{\D x^3}{\D y^a}\frac{\D x^3}{\D y^d}\frac{\D y^d}{\D x^3}\left(\frac{\D x^1}{\D y^b}\frac{\D y^b}{\D x^1}\frac{\D x^2}{\D y^c}\frac{\D y^c}{\D x^2}\frac{\D\hat{E}_2}{\D x^1}-\frac{\D x^2}{\D y^b}\frac{\D y^b}{\D x^2}\frac{\D x^1}{\D y^c}\frac{\D y^c}{\D x^1}\frac{\D\hat{E}_1}{\D x^2}\right)\\
=&g_{ad}\nb_b\varXi_c\left[\frac{\D y^d}{\D x^1}\left(\frac{\D y^b}{\D x^2}\frac{\D y^c}{\D x^3}-\frac{\D y^b}{\D x^3}\frac{\D y^c}{\D x^2}\right)+\frac{\D y^d}{\D x^2}\left(\frac{\D y^b}{\D x^3}\frac{\D y^c}{\D x^1}-\frac{\D y^b}{\D x^1}\frac{\D y^c}{\D x^3}\right)\right.\\
&\qquad\qquad+\left.\frac{\D y^d}{\D x^3}\left(\frac{\D y^b}{\D x^1}\frac{\D y^c}{\D x^2}-\frac{\D y^b}{\D x^2}\frac{\D y^c}{\D x^1}\right)\right]\\
=&g_{ad}\nb_b\varXi_c\det\left(\frac{\D(y^d,y^b,y^c)}{\D(x^1,x^2,x^3)}\right)
=g_{ad}\nb_b\varXi_c \eps^{dbc}\det\left(\frac{\D y}{\D x}\right)\\
=&(\curl \varXi)_a,
\end{align*}
where $\eps_{ijk}$ denotes the Levi-Civita symbol defined as follows:
\begin{align*}
\eps_{ijk} =
\begin{cases}
+1 & \text{if } (i,j,k) \text{ is } (1,2,3), (3,1,2) \text{ or } (2,3,1), \\
-1 & \text{if } (i,j,k) \text{ is } (1,3,2), (3,2,1) \text{ or } (2,1,3), \\
\;\;\,0 & \text{if }i=j \text{ or } j=k \text{ or } k=i,
\end{cases}
\end{align*}
which satisfies $\eps^{ijk}\equiv\eps_{ijk}$ in both Eulerian coordinates and  Lagrangian coordinates,
and the $i$th component of the curl of the vector $F$ reads
\begin{align*}
(\curl F)^i=\eps^{ijk}\D_j F_k, \text{ and } (\curl F)^i=\eps^{iab}\nb_a F_b,
\end{align*}
in Eulerian coordinates and Lagrangian coordinates, respectively.

Thus, we have obtained
\begin{align*}
D_t  \varpi_a=-(\curl \varXi)_a+u^b\nb_b\varpi_a+\varpi_b\nb_a u^b.
\end{align*}
Similarly, we get
\begin{align*}
\curl \varpi=0.
\end{align*}
We also have those equations for $D_t u_a$ and $D_t\beta_a$, one can see  \cite{HLarma} for details.

Thus, the system \eqref{mhd1} can be written in the Lagrangian coordinates, for $t>0$, as
\begin{subequations}\label{mhd41}
	\begin{numcases}{}
	D_tu_a+\nb_a  q^+  =u^c\nb_a u_c+\mu \beta^d\nb_d\beta_a, \quad\quad\quad \text{ in } \Omega^+,\label{mhd41.1}\\
	D_t\beta_a=\beta^d\nb_d u_a+\beta^c\nb_a u_c, \quad\quad\quad\quad\quad\quad\quad \text{ in }\Omega^+,\label{mhd41.2}\\
	D_t  \varpi_a=-(\curl \varXi)_a+u^b\nb_b\varpi_a+\varpi_b\nb_a u^b,\quad\; \text{ in } \Omega^-,\label{mhd41.7}\\
	\nb_a u^a=0 \text{ and } \nb_a\beta^a=0, \quad\quad\quad\quad\quad\quad\quad\quad \text{ in } \Omega^+,\label{mhd41.3}\\
	\nb_a u^a=0,\; \curl \varpi=0,\;  \nb_a\varpi^a=0, \text{ and } \nb_a\varXi^a=0, \;\text{ in } \Omega^-,\label{mhd41.3a}\\
	P=0, \quad \beta_aN^a=\varpi_aN^a=0, \quad\quad\quad\quad\quad\;\;\;\quad \text{ on }  \Gamma,\label{mhd41.4}\\
	N\times\varXi=u_N \varpi,\quad\quad\quad\quad\quad\quad\quad\quad\quad\quad\quad\quad\; \text{ on }  \Gamma,\label{mhd41.5}\\
	u=0,\;\varpi_aN^a=0, \; \varXi\times N=0, \quad\quad\quad\quad\quad\quad\; \text{ on }  W,\label{mhd41.6}
	\end{numcases}
\end{subequations}
where $N$ is the unit normal vector pointing into the interior of $\Omega^-$.

Finally, the energy defined by
\begin{align*}
E_0(t)=&\int_{\Omega^+} \left(\frac{1}{2}|u(t)|^2+\frac{\mu}{2}|\beta(t)|^2\right)d\mu_g
+\int_{\Omega^-}\frac{\mu}{2}|\varpi(t)|^2 d\mu_g
\end{align*}
is conserved. Of course, it is the equivalent one as in Eulerian coordinates. It can be easily verified by using the Gauss formula:
\begin{align}\label{Gauss}
\int_{\Omega^+} \nb_a F^a d\mu_g=\int_{\Gamma} N_a F^a d\mu_\gamma, \quad  \int_{\Omega^-} \nb_a F^a d\mu_g=-\int_{\Gamma\cup W} N_a F^a d\mu_\gamma,
\end{align}
where $F$ is a smooth vector-valued function, $N_a=g_{ab}N^b$  denotes the unit conormal, $g^{ab}N_aN_b=1$,  $N^a$ denotes the outward (or inward) unit normal to $\Gamma$ (and $W$) corresponding to $\Omega^+$ (or $\Omega^-$), $d\mu_\gamma$ is the volume element on boundaries, and the induced metric $\gamma$ on the tangent space to the boundary $T(\Gamma)$ (and $T(W)$) extended to be $0$ on the orthogonal complement in $T(\Omega^+)$ (and $\Omega^-$) is then given by
\begin{align*}
\gamma_{ab}=g_{ab}-N_aN_b,\quad \gamma^{ab}=g^{ab}-N^aN^b.
\end{align*}

The orthogonal projection of a $(r,s)$ tensor $S$ to the boundaries is given by
\begin{align*}
(\Pi S)_{b_1\cdots b_s}^{a_1\cdots a_r}=\gamma_{c_1}^{a_1}\cdots\gamma_{c_r}^{a_r}\gamma_{b_1}^{d_1}\cdots \gamma_{b_s}^{d_s}S_{d_1\cdots d_s}^{c_1\cdots c_r},
\end{align*}
where 
\begin{align}\label{gammaauc}
\gamma_a^c=\delta_a^c-N_aN^c.
\end{align}
Covariant differentiation on the boundary $\bnb$ is given by
$$\bnb S=\Pi\nb S,$$
and $\bnb$ is invariantly defined since the projection and the covariant derivatives are.
The second fundamental form of the boundary is given by
$$\theta_{ab}=(\bnb N)_{ab}=\gamma_a^c\nb_c N_b.$$
We need to extent the normal to a vector field defined and regular everywhere in the interior of $\Omega^+$ and $\Omega^-$ such  that when the geodesic distance to the boundaries $d(t,y)\ls \iota_0/4$, it is the normal to the set $\{y:\; d(t,y)=\text{const}\}$, and in the interior it drops off to $0$. We also denote the extension of the normal by $N^a$ which satisfies $|\nb N|\ls 2\norm{\theta}_{L^\infty(\Gamma\cup W)}$. Then we extent $\gamma$ to the pseudo-Riemannian metric $\gamma$ given by $\gamma_{ab}=g_{ab}-N_aN_b$ which satisfies $\nb\gamma|_{L^\infty(\Omega)}\ls C(\norm{\theta}_{L^\infty(\Gamma\cup W)}+1/\iota_0)$. One can see \cite{CL} for more details about the derivation of the gradient estimates of the extensions $N$ and $\gamma$.

\section{The Estimates of the magnetic field in vacuum}\label{sec.3}

Since $\Omega^-$ is simply connected, the equations $\nb_a\varpi^a=0$ and $\nb\times \varpi=0$ imply $\varpi(t,y)=\nb \varphi(t,y)$, where $\varphi$ is a solution of the Neumann problem
\begin{numcases}{}
\Delta\varphi=0,\quad\;\; \text{in } \Omega^-,\label{phi1}\\
\nb_N\varphi=0, \quad \text{on } \Gamma\cup W.\label{phi2}
\end{numcases}

For the derivatives of $\varpi$, we have the following $L^2$-estimates.

\begin{proposition}\label{prop.v.r}
	Let $r\gs 0$ be an integer. If $|\theta|+1/\iota_0\ls K$ on $\Gamma\cup W$, $|\nb N|\ls CK$ in $\Omega^-$, then it holds
	\begin{align*}
	\norm{\nb^{r+1}\varpi}_{L^2(\Omega^-)}^2\ls CK^{2(r+1)}E_0(0),\quad
	\norm{\nb^{r} \varpi}_{L^2(\Gamma\cup W)}^2
	\ls CK^{2r+1}E_0(0).
	\end{align*}
\end{proposition}

\begin{proof}
	We use the induction argument to show, for any integer $s\gs 0$, that
	\begin{align*}
	\norm{\nb^{s+1}\varpi}_{L^2(\Omega^-)}^2\ls CK^{2(s+1)}E_0(0),\quad
	\norm{\nb^{s} \varpi}_{L^2(\Gamma\cup W)}^2
	\ls CK^{2s+2}E_0(0).
	\end{align*}
	
	We first prove the case $s=0$. By \eqref{phi1}-\eqref{phi2} and  the H\"older inequality, we have
	\begin{align*}
	\norm{\nb^2\varphi}_{L^2(\Omega^-)}^2=&\int_{\Omega^-} g^{ab}g^{cd}\nb_c\nb_a\varphi \nb_d\nb_b\varphi d\mu_g\\
	=&\int_{\Omega^-} g^{ab}g^{cd}\nb_a(\nb_c\varphi \nb_d\nb_b\varphi) d\mu_g-\int_{\Omega^-} g^{cd}\nb_c\varphi \nb_d\Delta\varphi d\mu_g\\
	=&-\int_{\Gamma\cup W}N^bg^{cd}\nb_c\varphi \nb_d\nb_b\varphi d\mu_\gamma\\
	=&-\int_{\Gamma\cup W}g^{cd}\nb_c\varphi \nb_d(N^b\nb_b\varphi) d\mu_\gamma +\int_{\Gamma\cup W}g^{cd}\nb_c\varphi \nb_dN^b \nb_b\varphi d\mu_\gamma\\
	=&-\int_{\Gamma\cup W}\nb_N\varphi \nb_N^2\varphi d\mu_\gamma-\int_{\Gamma\cup W}\gamma^{cd}\nb_c\varphi \nb_d(\nb_N\varphi) d\mu_\gamma\\
	& +\int_{\Gamma\cup W}g^{cd}\nb_c\varphi \nb_dN^b \nb_b\varphi d\mu_\gamma\\
	=&\int_{\Gamma\cup W}\gamma^{cd}\nb_c\varphi \nb_dN^b \nb_b\varphi d\mu_\gamma\\
	\ls 
	&C\norm{\theta}_{L^\infty(\Gamma\cup W)}\norm{\varpi}_{L^2(\Gamma\cup W)}^2.
	\end{align*}
	We get, from Gauss' formula, H\"older's inequalities and Cauchy's inequality, that
	\begin{align}
	\norm{ \varpi}_{L^2(\Gamma\cup W)}^2=&\int_{\Gamma\cup W}\varpi^c\varpi_c d\mu_\gamma=\int_{\Gamma\cup W}N_aN^a\varpi^c\varpi_c d\mu_\gamma\no\\
	=&-\int_{\Omega^-} \nb_a(N^a\varpi^c\varpi_c) d\mu_g\no\\
	=&-\int_{\Omega^-} \nb_aN^a\varpi^c\varpi_c d\mu_g-2\int_{\Omega^-} N^a\nb_a\varpi^c\varpi_c d\mu_g\no\\
	\ls &C\norm{\tr(\nb N)}_{L^\infty(\Omega^-)}\norm{\varpi}_{L^2(\Omega^-)}^2+C\norm{\nb\varpi}_{L^2(\Omega^-)}\norm{\varpi}_{L^2(\Omega^-)}\label{v0.1}\\
	\ls &C\norm{\tr(\nb N)}_{L^\infty(\Omega^-)}\norm{\varpi}_{L^2(\Omega^-)}^2
	+\delta\norm{\nb\varpi}_{L^2(\Omega^-)}^2+C/\delta\norm{\varpi}_{L^2(\Omega^-)}^2,\no
	\end{align}
	for any $\delta>0$. Thus,  it follows that
	\begin{align*}
	\norm{\nb\varpi}_{L^2(\Omega^-)}^2\ls & C\norm{\theta}_{L^\infty(\Gamma\cup W)}\norm{\nb N}_{L^\infty(\Omega^-)}\norm{\varpi}_{L^2(\Omega^-)}^2
	+C\delta\norm{\theta}_{L^\infty(\Gamma\cup W)}\norm{\nb\varpi}_{L^2(\Omega^-)}^2\\
	&+C/\delta \norm{\theta}_{L^\infty(\Gamma\cup W)}\norm{\varpi}_{L^2(\Omega^-)}^2.
	\end{align*}
	Taking $\delta$ so small that $C\delta K=1/2$,  we get
	\begin{align*}
	\norm{\nb\varpi}_{L^2(\Omega^-)}^2\ls & C\norm{\theta}_{L^\infty(\Gamma\cup W)}\norm{\nb N}_{L^\infty(\Omega^-)}\norm{\varpi}_{L^2(\Omega^-)}^2
	+C/\delta \norm{\theta}_{L^\infty(\Gamma\cup W)}\norm{\varpi}_{L^2(\Omega^-)}^2\\
	\ls &CK^2\norm{\varpi}_{L^2(\Omega^-)}^2\ls CK^2E_0(0).
	\end{align*}
	We also have, with the help of \eqref{v0.1}, that
	\begin{align*}
	\norm{ \varpi}_{L^2(\Gamma\cup W)}^2\ls &CK\norm{\varpi}_{L^2(\Omega^-)}^2\ls CKE_0(0).
	\end{align*}

	Now, we assume that the claims are true for the cases $s= r-1$, then we prove the case $s=r$.
	
	From Gauss' formula, H\"older's inequality and Cauchy's inequality, it follows
	\begin{align}
	\norm{\nb^r\varpi}_{L^2(\Gamma\cup W)}^2=&\int_{\Gamma\cup W}N_aN^a|\nb^r\varpi|^2d\mu_\gamma
	=-\int_{\Omega^-}\nb_a(N^a|\nb^r\varpi|^2)d\mu_g\no\\
	\ls & CK\norm{\nb^r\varpi}_{L^2(\Omega^-)}^2+C\norm{\nb^r\varpi}_{L^2(\Omega^-)}\norm{\nb^{r+1}\varpi}_{L^2(\Omega^-)}.\label{v.r.0}
	\end{align}
	Since $\dv\varpi=0$ and $\nb\times\varpi=0$, we have, from \eqref{divcurl} and \eqref{divcurlint}, that
	\begin{align*}
	|\nb^{r+1}\varpi|^2\ls& C_1g^{bc}\gamma^{af}\gamma^{AF}\nb^{r}_A\nb_a\varpi_b\nb^{r}_F\nb_f\varpi_c,\\
	\int_{\Omega^-}|\nb^{r+1}\varpi|^2 d\mu_g\ls& C_2\int_{\Omega^-} (g^{bc}N^aN^f\gamma^{AF}\nb^{r}_A\nb_b\varpi_a\nb^{r}_F\nb_c\varpi_f +K^2|\nb^r\varpi|^2)d\mu_g.
	\end{align*}
	Noticing that $\nb\times\varpi=0$, it follows from the Gauss formula that
	\begin{align}
	&(C_1^{-1}+C_2^{-1})\norm{\nb^{r+1}\varpi}_{L^2(\Omega^-)}^2\no\\
	\ls &\int_{\Omega^-} g^{bc}\gamma^{af}\gamma^{AF}\nb^{r}_A\nb_a\varpi_b\nb^{r}_F\nb_c\varpi_f d\mu_g\no\\
	&+\int_{\Omega^-} (g^{bc}N^aN^f\gamma^{AF}\nb^{r}_A\nb_a\varpi_b\nb^{r}_F\nb_c\varpi_f +K^2|\nb^r\varpi|^2)d\mu_g\no\\
	=&\int_{\Omega^-} g^{bc}\nb_c(g^{af}\gamma^{AF}\nb^{r}_A\nb_a\varpi_b\nb^{r}_F\varpi_f) d\mu_g\label{v.r.1}\\
	&-\int_{\Omega^-} g^{bc}g^{af}\nb_c(\gamma^{AF})\nb^{r}_A\nb_a\varpi_b\nb^{r}_F\varpi_f d\mu_g\label{v.r.2}\\
	&+\int_{\Omega^-}K^2|\nb^r\varpi|^2d\mu_g.\label{v.r.3}
	\end{align}
	
	By the H\"older inequalities and the Cauchy inequality, we get
	\begin{align*}
	|\eqref{v.r.2}|\ls &CK\norm{\nb^r\varpi}_{L^2(\Omega^-)}\norm{\nb^{r+1}\varpi}_{L^2(\Omega^-)}.
	\end{align*}
	
	We write
	\begin{align}
	\eqref{v.r.1}=&-\int_{\Gamma\cup W} N^bg^{af}\gamma^{AF}\nb^{r}_A\nb_a\varpi_b\nb^{r}_F\varpi_f d\mu_\gamma\no\\
	=&-\int_{\Gamma\cup W} \gamma^{af}\nb_a(N^b\gamma^{AF}\nb^{r}_A\varpi_b\nb^{r}_F\varpi_f-N_fN^cN^b\gamma^{AF}\nb^{r}_A\varpi_b\nb^{r}_F\varpi_c) d\mu_\gamma\label{v.r.4.1}\\
	&+\int_{\Gamma\cup W} \gamma^{af}\nb_a(N^b\gamma^{AF})\nb^{r}_A\varpi_b\nb^{r}_F\varpi_f d\mu_\gamma\label{v.r.4.2}\\
	&+\int_{\Gamma\cup W} \gamma^{af}N^b\gamma^{AF}\nb^{r}_A\varpi_b\nb^{r}_F\nb_a\varpi_f d\mu_\gamma\label{v.r.4.3}\\
	&+\int_{\Gamma\cup W} N^aN^fN^b\gamma^{AF}\nb^{r}_A\varpi_b\nb^{r}_F\nb_a\varpi_f d\mu_\gamma\label{v.r.4.4}\\
	&-\int_{\Gamma\cup W} \gamma^{af}\nb_aN_fN^cN^b\gamma^{AF}\nb^{r}_A\varpi_b\nb^{r}_F\varpi_cd\mu_\gamma.\label{v.r.4.5}
	\end{align}
	In view of the Gauss formula, \eqref{v.r.4.1} vanishes. Due to $\dv\varpi=0$, $\eqref{v.r.4.3}+\eqref{v.r.4.4}=0$. From the H\"older inequality, we have
	\begin{align*}
	|\eqref{v.r.4.2}+\eqref{v.r.4.5}|\ls CK\norm{\nb^r\varpi}_{L^2(\Gamma\cup W)}^2.
	\end{align*}
	
	Thus, from \eqref{v.r.0} and the Cauchy inequality, we get
	\begin{align*}
	\norm{\nb^{r+1}\varpi}_{L^2(\Omega^-)}^2\ls& CK\norm{\nb^r\varpi}_{L^2(\Gamma\cup W)}^2+CK\norm{\nb^r\varpi}_{L^2(\Omega^-)}\norm{\nb^{r+1}\varpi}_{L^2(\Omega^-)}\\
	&+CK^2\norm{\nb^r\varpi}_{L^2(\Omega^-)}^2\\
	\ls &CK^2\norm{\nb^r\varpi}_{L^2(\Omega^-)}^2+CK\norm{\nb^r\varpi}_{L^2(\Omega^-)}\norm{\nb^{r+1}\varpi}_{L^2(\Omega^-)}\\
	\ls &CK^2\norm{\nb^r\varpi}_{L^2(\Omega^-)}^2+\frac{1}{2}\norm{\nb^{r+1}\varpi}_{L^2(\Omega^-)}^2,
	\end{align*}
	namely,
	\begin{align}\label{v.r1}
	\norm{\nb^{r+1}\varpi}_{L^2(\Omega^-)}^2\ls CK^2\norm{\nb^r\varpi}_{L^2(\Omega^-)}^2,
	\end{align}
	and then
	\begin{align}\label{v.r}
	\norm{\nb^r\varpi}_{L^2(\Gamma\cup W)}^2
	\ls & CK\norm{\nb^r\varpi}_{L^2(\Omega^-)}^2.
	\end{align}
	Therefore, by the induction argument, we have obtained the desired results.
\end{proof}

\begin{proposition}\label{prop.vr2}
	Suppose that for $\iota_1\gs 1/K_1$,
	\begin{align*}
	\abs{N(\bar{x}_1)-N(\bar{x}_2)}\ls \eps_1, \quad \text{whenever } |\bar{x}_1-\bar{x}_2|\ls \iota_1, \; \bar{x}_1,\bar{x}_2\in\Gamma\cup W,
	\end{align*}
	and
	\begin{align*}
	C_0^{-1}\gamma_{ab}^0(y) Z^aZ^b\ls \gamma_{ab}(t,y)Z^aZ^b\ls C_0\gamma_{ab}^0(y) Z^aZ^b, \quad \text{if } Z\in T(\Omega^-),
	\end{align*}
	where $\gamma_{ab}^0(y)=\gamma_{ab}(0,y)$. Then, it holds for any integer $r\gs 0$
	\begin{align*}
	\norm{\nb^r\varpi}_{L^\infty(\Gamma\cup W)}\ls C(r,K,K_1,\vol\Omega^-)E^{1/2}_0(0),
	\end{align*}
	and
	\begin{align*}
	\norm{\Pi \nb^r |\varpi|^2}_{L^2(\Gamma\cup W)}\ls &C(r,K,K_1,\vol\Omega^-) E_0(0).
	\end{align*}
\end{proposition}

\begin{proof}
	From \eqref{A.32} and \eqref{CLA.7.1}, it follows
	\begin{align*}
	\norm{\nb^s\varpi}_{L^\infty(\Gamma\cup W)}\ls &C\norm{\nb^{s+1}\varpi}_{L^4(\Gamma\cup W)}+C(K_1)\norm{\nb^s\varpi}_{L^4(\Gamma\cup W)}\\
	\ls &C(K,K_1,\vol\Omega^-)\sum_{\ell=0}^{2}\norm{\nb^{s+\ell}\varpi}_{L^2(\Omega^-)}\\
	\ls &C(K,K_1,\vol\Omega^-)\sum_{\ell=0}^{2}K^{s+\ell}E^{1/2}_0(0)\\
	\ls &C(s,K,K_1,\vol\Omega^-)E^{1/2}_0(0).
	\end{align*}
	By the H\"older inequality, Proposition~\ref{prop.v.r}, \eqref{A.32},  and the Cauchy inequality, we have
	\begin{align*}
	&\norm{\Pi \nb^r |\varpi|^2}_{L^2(\Gamma\cup W)}\ls C\sum_{m=0}^{[r/2]} \norm{\nb^{m}\varpi}_{L^\infty(\Gamma\cup W)}\norm{\nb^{r-m}\varpi}_{L^2(\Gamma\cup W)}\\
	\ls &C\sum_{m=0}^{[r/2]}(\norm{\nb^{m+1}\varpi}_{L^4(\Gamma\cup W)}+C(K_1)\norm{\nb^{m}\varpi}_{L^4(\Gamma\cup W)})K^{r-m+1/2}E^{1/2}(0)\\
	\ls &C(K_1,\vol\Omega^-)E^{1/2}(0)\sum_{m=0}^{[r/2]}\sum_{\ell=0}^{2}\norm{\nb^{m+\ell}\varpi}_{L^2(\Omega^-)}K^{r-m+1/2}\\
	\ls &C(K_1,\vol\Omega^-)\sum_{m=0}^{[r/2]}\sum_{\ell=0}^{2}K^{m+\ell}K^{r-m+1/2}E_0(0)\\
	\ls &C(r,K,K_1,\vol\Omega^-)E_0(0).
	\end{align*}	
\end{proof}

\section{The estimates of the electric field in vacuum}\label{sec.4}

Although the electric field can be regarded as a secondary variable due to $$\nb\times \hat{\el}=-\hat{\vH}_t, \text{ and } \dv\hat{\el}=0,$$ 
we have to use the estimates of the electric field in vacuum in order to get the energy estimates.  In fact, we can prove the following estimates.

\begin{proposition}\label{prop.E.r}
	If $|\theta|+1/\iota_0\ls K$ on $\Gamma\cup W$ and $\nb\times\varXi_0\in L^2(\Omega^-)$, then it holds for any integer $r\gs 0$ 
	\begin{align*}
	&\norm{\nb^r (\nabla\times\varXi)}_{L^2(\Omega^-)}^2+\norm{\nb^r (\nabla\times\varXi)}_{L^2(\Gamma\cup W)}^2\no\\
	\ls&C(r,K)\left[\norm{\nb\times\varXi_0}_{L^2(\Omega^-)}^2+ E_0(0)\sup_{\tau\in [0,t]}\sum_{\ell=0}^{2}\norm{\nb^\ell u(\tau)}_{L^2(\Gamma)}^2\right]\no\\
	&\qquad\qquad\qquad\cdot \exp\left(C(K)\int_0^t\norm{u_N(\tau)}_{L^\infty(\Gamma)}d\tau\right),
	\end{align*}
	where $\varXi_0$ is the initial datum of $\varXi$ at $t=0$.
\end{proposition}

\begin{proof}
	For convenience, we denote $B=\nabla\times\varXi$ in this section. Then, in Lagrangian coordinates, we have from \eqref{mhd41.7} that
	\begin{align*}
	B_a=-D_t\varpi_a+u^b\nb_b\varpi_a+\varpi_b\nb_a u^b.
	\end{align*}
	From \eqref{CL3.9.1}, we have on boundaries $\Gamma\cup W$
	\begin{align*}
	N^aB_a=&-D_t(N^a\varpi_a)+D_tN^a\varpi_a+N^au^b\nb_b\varpi_a+N^a\varpi_b\nb_a u^b\\
	=&-2h^a_dN^d\varpi_a-h_{NN}N^a\varpi_a+N^au^b\nb_b\varpi_a+N^a\varpi_b\nb_a u^b\\
	=&-(\nb_au_d+\nb_du_a)N^d\varpi^a+N^au^b\nb_b\varpi_a+N^a\varpi_b\nb_a u^b\\
	=&-\nb_au_dN^d\varpi^a+N^du^a\nb_a\varpi_d.
	\end{align*}
	By \eqref{CL2.4.2} and the fact $\nb_c\varpi^c=0$, we also get in $\Omega^-$
	\begin{align*}
	\nb_a B^a=&g^{ac}\nb_cB_a=-g^{ac}\nb_cD_t\varpi_a+\nb_c(u^b\nb_b\varpi^c)+g^{ac}\nb_c(\varpi_b\nb_a u^b)\\
	=&[D_t,g^{ac}\nb_c]\varpi_a-D_t(\nb_c\varpi^c)+\nb_cu^b\nb_b\varpi^c+g^{ac}\nb_c\varpi_b\nb_a u^b+\varpi_b\Delta u^b\\
	=&-2h^{ac}\nb_c\varpi_a-(\Delta u^e)\varpi_e+2h^{ac}\nb_c\varpi_a+\varpi_b\Delta u^b\\
	=&0,
	\end{align*}
	and by \eqref{CL2.4.1} and the fact $\nb\times\varpi=0$,
	\begin{align*}
	(\nb\times B)^c=&\eps^{cea}\nb_eB_a=-\eps^{cea}\nb_eD_t\varpi_a+\eps^{cea}\nb_e(u^b\nb_b\varpi_a)+\eps^{cea}\nb_e(\varpi_b\nb_a u^b)\\
	=&\eps^{cea}[D_t,\nb_e]\varpi_a-\eps^{cea}D_t\nb_e\varpi_a+\eps^{cea}\nb_eu^b\nb_b\varpi_a+\eps^{cea}\nb_e\varpi_b\nb_a u^b\\
	=&-\eps^{cea}\nb_a\nb_e u^d\varpi_d+\eps^{cea}\nb_eu^b\nb_b\varpi_a-\eps^{cea}\nb_a\varpi_b\nb_e u^b\\
	=&0.
	\end{align*}
	Thus, we have $\nb_aB^a=0$ and $\nb\times B=0$, which yields $B(t,y)=\nb \psi(t,y)$, and $\psi$ is a solution of the following Neumann problem
	\begin{numcases}{}
	\Delta \psi=0, \quad\;\; \text{in } \Omega^-,\label{psi.1}\\
	\nb_N\psi =f, \quad \text{on } \Gamma,\label{psi.2}\\
	\nb_N\psi=0, \quad \text{on } W,\label{psi.3}
	\end{numcases}
	where $f=N^du^a\nb_a\varpi_d-\nb_au_dN^d\varpi^a$.
	
	From \eqref{psi.1}-\eqref{psi.3} and H\"older's inequality, we get
	\begin{align*}
	&\norm{\bnb B}_{L^2(\Omega^-)}^2=\int_{\Omega^-} g^{ab}\gamma^{cd}\nb_c\nb_a\psi\nb_d\nb_b\psi d\mu_g\\
	=&\int_{\Omega^-}g^{ab}\nb_a(\gamma^{cd}\nb_c\psi\nb_d\nb_b\psi) d\mu_g-\int_{\Omega^-}g^{ab}\nb_a\gamma^{cd}\nb_c\psi\nb_d\nb_b\psi d\mu_g\\
	=&-\int_{\Gamma\cup W}N^b\gamma^{cd}\nb_c\psi\nb_d\nb_b\psi d\mu_\gamma-\int_{\Omega^-}g^{ab}\nb_a\gamma^{cd}\nb_c\psi\nb_d\nb_b\psi d\mu_g\\
	=&-\int_{\Gamma\cup W}\gamma^{cd}\nb_c\psi\nb_d(N^b\nb_b\psi) d\mu_\gamma+\int_{\Gamma\cup W}\gamma^{cd}\nb_c\psi\nb_d N^b\nb_b\psi d\mu_\gamma\\
	&-\int_{\Omega^-}g^{ab}\nb_a\gamma^{cd}\nb_c\psi\nb_d\nb_b\psi d\mu_g\\
	\ls &C\norm{\Pi B}_{L^2(\Gamma)}\norm{\bnb f}_{L^2(\Gamma)}+C\norm{\theta}_{L^\infty(\Gamma\cup W)}\norm{B}_{L^2(\Gamma\cup W)}^2+CK\norm{B}_{L^2(\Omega^-)}\norm{\nb B}_{L^2(\Omega^-)}.
	\end{align*}
	Similar to \eqref{v0.1}, we have for any $\delta>0$
	\begin{align}
	\norm{B}_{L^2(\Gamma\cup W)}^2
	\ls &C(K+1/\delta)\norm{B}_{L^2(\Omega^-)}^2
	+\delta\norm{\nb B}_{L^2(\Omega^-)}^2.\label{Bb}
	\end{align}
	Thus, in view of \eqref{divcurl} and Cauchy's inequality, it follows that
	\begin{align*}
	\norm{\nb B}_{L^2(\Omega^-)}^2\ls &C\norm{\bnb B}_{L^2(\Omega^-)}^2\\
	\ls& C\norm{\bnb f}_{L^2(\Gamma)}^2+C(K+1)(K+1/\delta)\norm{B}_{L^2(\Omega^-)}^2
	\\
	&+C(K+1)\delta\norm{\nb B}_{L^2(\Omega^-)}^2+CK^2\norm{B}_{L^2(\Omega^-)}^2+\frac{1}{4}\norm{\nb B}_{L^2(\Omega^-)}^2.
	\end{align*}
	Taken $\delta$ so small that $C(K+1)\delta<1/4$, it yields
	\begin{align}
	\norm{\nb B}_{L^2(\Omega^-)}^2\ls &C\norm{\bnb f}_{L^2(\Gamma)}^2+C(K)\norm{B}_{L^2(\Omega^-)}^2.\label{nB}
	\end{align}
	
	Similar to the derivation of \eqref{mhd41.7}, we can get, due to $\hat{\el}_t=0$, that
	\begin{align}
	D_t B_a=u^b\nb_b B_a+\nb_a u^b B_b. \label{DtB}
	\end{align}
	It follows that
	\begin{align*}
	\frac{d}{dt}\int_{\Omega^-}|B|^2d\mu_g=-\int_{\Gamma} u_N |B|^2d\mu_\gamma,
	\end{align*}
	which yields
	\begin{align}
	\norm{B(t)}_{L^2(\Omega^-)}^2\ls \norm{B_0}_{L^2(\Omega^-)}^2+\int_0^t\norm{u_N(\tau)}_{L^\infty(\Gamma)}\norm{B(\tau)}_{L^2(\Gamma)}^2d\tau,\label{BL2}
	\end{align}
	where $B_0=B(t)|_{t=0}$. From  \eqref{nB}, \eqref{BL2} and \eqref{Bb}, we obtain
	\begin{align*}
	&\norm{B(t)}_{L^2(\Omega^-)}^2+\norm{\nb B}_{L^2(\Omega^-)}^2\\
	\ls &C\norm{\bnb f}_{L^2(\Gamma)}^2+C(K)\norm{B_0}_{L^2(\Omega^-)}^2+C(K)\int_0^t\norm{u_N(\tau)}_{L^\infty(\Gamma)}\norm{B(\tau)}_{L^2(\Gamma)}^2d\tau\\
	\ls &C\norm{\bnb f}_{L^2(\Gamma)}^2+C(K)\norm{B_0}_{L^2(\Omega^-)}^2\\
	&+C(K)\int_0^t\norm{u_N(\tau)}_{L^\infty(\Gamma)}\big(\norm{B(\tau)}_{L^2(\Omega^-)}^2
	+\norm{\nb B(\tau)}_{L^2(\Omega^-)}^2\big)d\tau,
	\end{align*}
	which implies, by Gr\"onwall's inequality, that
	\begin{align*}
	&\norm{B(t)}_{L^2(\Omega^-)}^2+\norm{\nb B}_{L^2(\Omega^-)}^2\no\\
	\ls&\left[ C\sup_{\tau\in [0,t]}\norm{\bnb f(\tau)}_{L^2(\Gamma)}^2+C(K)\norm{B_0}_{L^2(\Omega^-)}^2\right]\exp\left(C(K)\int_0^t\norm{u_N(\tau)}_{L^\infty(\Gamma)}d\tau\right).
	\end{align*}
	
	By the definition of $f$ and Proposition \ref{prop.v.r}, we have
	\begin{align*}
	\norm{\bnb f}_{L^2(\Gamma)}\ls &\norm{\bnb (N^du^a\nb_a\varpi_d-\nb_au_dN^d\varpi^a)}_{L^2(\Gamma)}\\
	\ls &\norm{\theta}_{L^\infty(\Gamma)}\norm{u}_{L^2(\Gamma)}\norm{\nb\varpi}_{L^2(\Gamma)}+\norm{\nb u}_{L^2(\Gamma)}\norm{\nb\varpi}_{L^2(\Gamma)}\\
	&+\norm{ u}_{L^2(\Gamma)}\norm{\nb^2\varpi}_{L^2(\Gamma)}+\norm{ \nb^2 u}_{L^2(\Gamma)}\norm{\varpi}_{L^2(\Gamma)}\\
	&+\norm{\theta}_{L^\infty(\Gamma)}\norm{\nb u}_{L^2(\Gamma)}\norm{\varpi}_{L^2(\Gamma)}\\
	\ls &C(K)\left[\norm{u}_{L^2(\Gamma)}+\norm{\nb u}_{L^2(\Gamma)}+\norm{ \nb^2 u}_{L^2(\Gamma)}\right]E^{1/2}(0)\\
	\ls &C(K)E^{1/2}(0)\sum_{\ell=0}^{2}\norm{\nb^\ell u}_{L^2(\Gamma)}.
	\end{align*}
	Therefore, we obtain,combining with \eqref{Bb}, that
	\begin{align*}
	&\norm{B}_{L^2(\Omega^-)}^2+\norm{\nb B}_{L^2(\Omega^-)}^2+\norm{B}_{L^2(\Gamma\cup W)}^2\no\\
	\ls&C(K)\left[\norm{B_0}_{L^2(\Omega^-)}^2+ E_0(0)\sup_{\tau\in [0,t]}\sum_{\ell=0}^{2}\norm{\nb^\ell u(\tau)}_{L^2(\Gamma)}^2\right]\no\\
	&\qquad\qquad\qquad\cdot \exp\left(C(K)\int_0^t\norm{u_N(\tau)}_{L^\infty(\Gamma)}d\tau\right).
	\end{align*}
	
	Since $\nb_a B^a=0$ and $\nb\times B=0$, it is similar to $\varpi$.  One can verify that the lines between \eqref{v.r.0} and \eqref{v.r} also hold if $\varpi$ is replaced by $B$ everywhere in those lines. Thus, we can obtain for any $r\gs 1$
	\begin{align}\label{B.r1}
	\norm{\nb^{r+1}B}_{L^2(\Omega^-)}^2\ls CK^2\norm{\nb^rB}_{L^2(\Omega^-)}^2,
	\end{align}
	and 
	\begin{align}\label{B.r}
	\norm{\nb^rB}_{L^2(\Gamma\cup W)}^2
	\ls & CK\norm{\nb^rB}_{L^2(\Omega^-)}^2.
	\end{align}
	Hence, we get for any $r\gs 0$
	\begin{align*}
	&\norm{\nb^r B}_{L^2(\Omega^-)}^2+\norm{\nb^r B}_{L^2(\Gamma\cup W)}^2\no\\
	\ls&C(K,\vol\Omega^+)\left[\norm{B_0}_{L^2(\Omega^-)}^2+ E_0(0)\sup_{\tau\in [0,t]}\sum_{\ell=0}^{3}\norm{\nb^\ell u(\tau)}_{L^2(\Omega^+)}^2\right]\no\\
	&\qquad\qquad\qquad\cdot \exp\left(C(K)\int_0^t\norm{u_N(\tau)}_{L^\infty(\Gamma)}d\tau\right).
	\end{align*}
	Changing $B$ back to $\nb\times\varXi$, we obtain the desired results.
\end{proof}

\section{The General $r$-th Order Energy Estimates}\label{sec.5}

We can get   that for $r\gs 1$ (cf. \cite{HLarma})
\begin{align}\label{r.u}
&D_t\nb^r u_a+\nb^r\nb_a q^+ \no\\
=&(\nb_au_c-\sgn(r-1)\nb_cu_a)\nb^r u^c+\mu\beta^c\nb_c\nb^r\beta_a+r\mu\nb\beta\cdot\nb^r\beta_a\no\\
&+\sgn(r-1)\mu\nb^r\beta^c\nb_c\beta_a
+\sgn((r-1)(r-2))\mathcal{P}_a(\beta),
\end{align}
and
\begin{align}\label{nb.beta}
D_t\nb^r \beta_a=&(\nb_au_c+\sgn(r-1)\nb_c u_a)\nb^r\beta^c-\sgn(r-1)\nb^r u^c\nb_c\beta_a\no\\
&+\beta^c\nb_c\nb^r u_a+r\nb\beta\cdot\nb^r u_a+\sgn((r-1)(r-2))\mathcal{Q}_a.
\end{align}
where 
\begin{align*}
\mathcal{P}_a(\beta):=-\sum_{s=2}^{r-1}\bic_r^s \nb^su^c\nb^{r-s} \nb_cu_a+\mu\sum_{s=2}^{r-1}\bic_r^s \nb^s\beta^c\nb^{r-s} \nb_c\beta_a,
\end{align*}
and
\begin{align*}
\mathcal{Q}_a:=-\sum_{s=2}^{r-1}\bic_r^s \nb^su^c\nb^{r-s} \nb_c \beta_a+\sum_{s=2}^{r-1}\bic_r^s \nb^s\beta^c\nb^{r-s} \nb_c u_a.
\end{align*}

Define the $r$-th order energy as
\begin{align*}
E_r(t)=&\int_{\Omega^+} g^{bd}\gamma^{af}\gamma^{AF}\nb_A^{r-1}\nb_a u_b\nb_F^{r-1}\nb_f u_d d\mu_g\\
&+\mu \int_{\Omega^+} g^{bd}\gamma^{af}\gamma^{AF}\nb_A^{r-1}\nb_a \beta_b\nb_F^{r-1}\nb_f \beta_d d\mu_g\\
&+\int_{\Omega^+} |\nb^{r-1}\curl u|^2 d\mu_g+\mu \int_{\Omega^+} |\nb^{r-1}\curl \beta|^2 d\mu_g\\
&+\int_{\Gamma} \gamma^{af}\gamma^{AF}\nb_A^{r-1}\nb_a P\nb_F^{r-1}\nb_f P\, \vartheta d\mu_\gamma,
\end{align*}
where $\vartheta=-1/\nb_N P$ as before.

\begin{theorem}\label{thm.renergy}
	Let $1\ls r\ls 4$ be an integer, then there exists a $T>0$ such that the following holds: For any smooth solution of MHD \eqref{mhd41}  satisfying
	\begin{align}
	|\beta|\ls M_1 \quad\text{for } r=& 2,&&\text{in } [0,T]\times \Omega^+,\label{2energy81}\\
	|\nb P|+|\nb u|+|\nb \beta|\ls &M,  &&\text{in } [0,T]\times \Omega,\label{2energy8}\\
	|\theta|+1/\iota_0\ls &K,&&\text{on } [0,T]\times \Gamma,\label{2energy9}\\
	-\nb_N P\gs \eps>&0, &&\text{on } [0,T]\times \Gamma,\label{2energy91}\\
	|u|+|\nb^2 P|+|\nb_ND_t P|\ls& L,&&\text{on } [0,T]\times \Gamma,\label{2energy92}
	\end{align}
	we have, for $t\in[0,T]$,
	\begin{align}\label{renergy}
	E_r(t)\ls e^{C_1t}E_r(0)+C_2e^{C_3t}\left(e^{C_4t}-1\right),
	\end{align}
	where the constant $C_i>0$ depends on $K$, $K_1$, $M$, $M_1$, $L$, $1/\eps$, $\vol\Omega$,  $E_0(0)$, $E_1(0)$, $\cdots$, and $E_{r-1}(0)$; both $C_2$ and $C_3$ also depend on $\norm{\curl \varXi_0}_{L^2(\Omega^-)}$ if $r\gs 3$; $C_3=0$ for $r=1,2$.
\end{theorem}

\begin{proof} Since $\tr h=0$, we have
	\begin{align}
	&\frac{d}{dt}E_r(t)=\no\\
	&\quad\int_{\Omega^+} D_t\left(g^{bd}\gamma^{af}\gamma^{AF}\nb_A^{r-1}\nb_a u_b\nb_F^{r-1}\nb_f u_d\right) d\mu_g\label{r.e1}\\
	&+\mu \int_{\Omega^+} D_t\left(g^{bd}\gamma^{af}\gamma^{AF}\nb_A^{r-1}\nb_a \beta_b\nb_F^{r-1}\nb_f \beta_d\right) d\mu_g\label{r.e2}\\
	&+\int_{\Omega^+} D_t|\nb^{r-1}\curl u|^2 d\mu_g
	+\mu \int_{\Omega^+} D_t|\nb^{r-1}\curl \beta|^2 d\mu_g\label{r.e3}\\
	&+\int_{\Gamma} D_t\left(\gamma^{af}\gamma^{AF}\nb_A^{r-1}\nb_a P\nb_F^{r-1}\nb_f P\right)\, \vartheta d\mu_\gamma\label{r.e7}\\
	&+\int_{\Gamma} \gamma^{af}\gamma^{AF}\nb_A^{r-1}\nb_a P\nb_F^{r-1}\nb_f P\left(\frac{\vartheta_t}{\vartheta}-h_{NN}\right)\, \vartheta d\mu_\gamma. \label{r.e8}
	\end{align}
	
	Since the boundary integrals disappear for the case $r=1$, it is easy to obtain the desired estimate and we omit the details. So we assume $r\gs 2$ from now on in the proofs.
	
	We first estimate  \eqref{r.e1}--\eqref{r.e2} and \eqref{r.e7}. 
	From Lemmas \ref{lem.CL2.1} and \ref{lem.CL3.9}, 
	we have in $\Omega^+$,
	\begin{align*}
	&D_t\left(g^{bd}\gamma^{af}\gamma^{AF}\nb_A^{r-1}\nb_a u_b\nb_F^{r-1}\nb_f u_d\right)\\
	=&-2\nb_c u_e\gamma^{af}\gamma^{AF}\nb_A^{r-1}\nb_a u^c\nb_F^{r-1}\nb_f u^e-4r\nb_c u_e\gamma^{ac}\gamma^{ef}\gamma^{AF}\nb_A^{r-1}\nb_a u^d\nb_F^{r-1}\nb_f u_d\\
	&-2\gamma^{af}\gamma^{AF}\nb_F^{r-1}\nb_f u^b\nb_A^{r-1}\nb_a\nb_b q^+ \\ &+2\gamma^{af}\gamma^{AF}\nb_F^{r-1}\nb_f u^b(\nb_bu_c-\sgn(r-1)\nb_cu_b)\nb_A^{r-1}\nb_a u^c\\
	&+2\mu\gamma^{af}\gamma^{AF}\nb_F^{r-1}\nb_f u_d\left(\beta^c\nb_c\nb_{Aa}^r\beta^d\right)+2r\mu\gamma^{af}\gamma^{AF}\nb_F^{r-1}\nb_f u_d(\nb\beta^c\nb^{r-1}\nb_c\beta^d)_{Aa}\\
	&+2\sgn(r-1)\mu \gamma^{af}\gamma^{AF}\nb_F^{r-1}\nb_f u_d \nb_{Aa}^r\beta^c\nb_c\beta_b\\
	&+2\sgn((r-1)(r-2))\gamma^{af}\gamma^{AF}\nb_F^{r-1}\nb_f u_d (\mathcal{P}_b(\beta))_{Aa},
	\end{align*}
	and
	\begin{align*}
	&D_t\left(g^{bd}\gamma^{af}\gamma^{AF}\nb_A^{r-1}\nb_a \beta_b\nb_F^{r-1}\nb_f \beta_d\right)\\
	=&-2\nb_c u_e\gamma^{af}\gamma^{AF}\nb_A^{r-1}\nb_a \beta^c\nb_F^{r-1}\nb_f \beta^e-2r\nb_c u_e\gamma^{ac}\gamma^{ef}\gamma^{AF}\nb_A^{r-1}\nb_a \beta^d\nb_F^{r-1}\nb_f \beta_d\\
	&+2\gamma^{af}\gamma^{AF}\nb_F^{r-1}\nb_f \beta^b (\nb_bu_c+\sgn(r-1)\nb_cu_b)\nb_A^{r-1}\nb_a\beta^c\\
	&-2\sgn(r-1)\gamma^{af}\gamma^{AF}\nb_F^{r-1}\nb_f \beta^b\nb_c\beta_b\nb_A^{r-1}\nb_a u^c\\
	&+2\gamma^{af}\gamma^{AF}\nb_F^{r-1}\nb_f \beta^d \beta^c\nb_c\nb_{Aa}^r u_d\\
	&+2r\gamma^{af}\gamma^{AF}\nb_F^{r-1}\nb_f \beta^d (\nb\beta\cdot\nb^ru_b)_{Aa}\\
	&+2\sgn((r-1)(r-2))\gamma^{af}\gamma^{AF}\nb_F^{r-1}\nb_f \beta^d (\mathcal{Q}_b)_{Aa},
	\end{align*}
	and in the interface $\Gamma$,
	\begin{align*}
	D_t\left(\gamma^{af}\gamma^{AF}\nb_{Aa}^r P\nb_{Ff}^r P\right)
	=&-2r\nb_c u_e\gamma^{ac}\gamma^{ef}\gamma^{AF}\nb_{Aa}^r P\nb_{Ff}^r P+2 \gamma^{af}\gamma^{AF}\nb_{Aa}^r P D_t\nb_{Ff}^r P.
	\end{align*}
	
	Thus, we get
	\begin{align}
	&\eqref{r.e1}+\eqref{r.e2}+\eqref{r.e7}\no\\
	\ls& C\left(\norm{\nb u}_{L^\infty(\Omega)}+\norm{\nb \beta}_{L^\infty(\Omega^+)}\right) E_r(t)\no\\
	&+C E_r^{1/2}(t)\sum_{s=2}^{r-1}\left(\norm{\nb^s u}_{L^4(\Omega^+)}+\norm{\nb^s \beta}_{L^4(\Omega^+)}\right)\no\\
	&\qquad\qquad\qquad\quad\cdot\left(\norm{\nb^{r-s+1}u}_{L^4(\Omega^+)}+\norm{\nb^{r-s+1} \beta}_{L^4(\Omega^+)}\right)\label{r.e112}\\
	&+2\int_{\Gamma}\gamma^{af}\gamma^{AF}\nb_{Aa}^{r}P\left(D_t\nb_{Ff}^{r} P-\frac{1}{\vartheta}N_b\nb_{Ff}^{r} u^b\right) \vartheta d\mu_\gamma \label{r.e11}\\
	&-\mu\int_{\Gamma}\gamma^{af}\gamma^{AF}\nb_{Aa}^{r}|\varpi|^2N_b\nb^r_{Ff} u^b d\mu_\gamma\label{r.e15}\\
	&+2\int_{\Omega^+}\nb_b\left(\gamma^{af}\gamma^{AF}\right)\nb_{Ff}^{r} u^b\nb_{Aa}^{r}q^+d\mu_g\label{r.e12}\\
	&+2\mu \int_{\Gamma} N_c\gamma^{af}\gamma^{AF}\nb_F^{r-1}\nb_f u_d\beta^c\nb_{Aa}^r\beta^d d\mu_\gamma\label{r.e13}\\
	&-2\mu \int_{\Omega^+} \nb_c\left(\gamma^{af}\gamma^{AF}\right)\nb_F^{r-1}\nb_f u_d\beta^c\nb_{Aa}^r\beta^d d\mu_g.\label{r.e14}
	\end{align}
	
	From Lemma \ref{lem.CLA.3}, it follows that
	\begin{align*}
	|\eqref{r.e112}|\ls C(K,K_1,M,\vol\Omega,1/\eps)\left(1+\sum_{s=0}^{r-1} E_s(t)\right)E_r(t).
	\end{align*}
	By Gauss' formula, Proposition \ref{prop.vr2} and \eqref{CLA.7.1}, we get
	\begin{align*}
	\eqref{r.e15}=&-\mu\int_{\Gamma}\gamma^{af}\nb_f[\gamma^{AF}\nb_{Aa}^{r}|\varpi|^2N_b\nb^{r-1}_{F} u^b-N_aN^c\gamma^{AF}\nb_{Ac}^{r}|\varpi|^2N_b\nb^{r-1}_{F} u^b] d\mu_\gamma\\
	&+\mu\int_{\Gamma}\gamma^{af}\nb_f[\gamma^{AF}\nb_{Aa}^{r}|\varpi|^2N_b]\nb^{r-1}_{F} u^b d\mu_\gamma\\
	&-\mu\int_{\Gamma}\gamma^{af}\nb_fN_aN^c\gamma^{AF}\nb_{Ac}^{r}|\varpi|^2N_b\nb^{r-1}_{F} u^b d\mu_\gamma\\
	\ls &\mu C(r,K,K_1,\vol\Omega^-)E_0(0)\norm{\nb^{r-1}u}_{L^2(\Gamma)}\\
	\ls &\mu C(r,K,K_1,\vol\Omega^-,\vol\Omega^+)E_0(0)(\norm{\nb^r u}_{L^2(\Omega^+)}+\norm{\nb^{r-1} u}_{L^2(\Omega^+)})\\
	\ls &\mu C(r,K,K_1,\vol\Omega)E_0(0)\left(E_r^{1/2}(t)+E_{r-1}^{1/2}(t)\right).
	\end{align*}
	
	Due to $\beta\cdot N=0$ on $\Gamma$,  \eqref{r.e13} vanishes. 
	
	From \eqref{A.32} and \eqref{CLA.7.1}, it follows
	\begin{align*}
	\norm{u}_{L^\infty(\Gamma)}\ls C(K_1)\sum_{s=0}^2  \norm{\nb^s u}_{L^2(\Omega^+)}\ls C(K_1)\sum_{s=0}^2 E_s^{1/2}(t).
	\end{align*}
	%
	%
	
	From Lemma \ref{lem.CLA.4}, it follows, for $\iota_1\gs 1/K_1$, that
	\begin{align}\label{CllemmaA.4}
	\norm{\beta}_{L^\infty(\Omega^+)}\ls C\sum_{0\ls s\ls 2} K_1^{n/2-s} \norm{\nb^s \beta}_{L^2(\Omega^+)}\ls C(K_1)\sum_{s=0}^2 E_s^{1/2}(t).
	\end{align}
	Thus, with the help of the H\"older inequality, we have for any $r\gs 3$
	\begin{align*}
	\eqref{r.e14}\ls &CK\norm{\beta}_{L^\infty(\Omega^+)}E_r(t)\ls C(K,K_1) \left(\sum_{s=0}^2 E_s^{1/2}(t)\right)E_r(t).
	\end{align*}
	For $r=1$, it is easy to verify that there exists a $T>0$ such that $E_1(t)$ can be controlled by the initial energy $E_1(0)$ for $t\in[0,T]$, e.g., $E_1(t)\ls 2E_1(0)$.
	For $r=2$, we have to assume the a priori bound $|\beta|\ls M_1$ on $[0,T]\times \Omega^+$, i.e., \eqref{2energy81}, in order to get a bound that is linear in the highest-order derivative or energy. Then, we have from \eqref{2energy81} for $r=2$
	\begin{align*}
	\eqref{r.e14}\ls &CK\norm{\beta}_{L^\infty(\Omega^+)}E_r(t)\ls C(K,M_1) E_r(t).
	\end{align*}
	
	From the H\"older inequality, we get
	\begin{align}
	\eqref{r.e12}\ls CKE_r^{1/2}(t)\norm{\nb^r q^+}_{L^2(\Omega^+)}.\label{up}
	\end{align}
	
	From \eqref{mhd1}, it follows
	\begin{align*}
	\D_j(D_t v^j)+\Delta q^+ =\mu \D_j(H^k\D_k H^j),
	\end{align*}
	which yields from \eqref{DtDi}
	\begin{align*}
	\Delta q^+=-\D_j v^k\D_k v^j+\mu \D_j H^k\D_k H^j.
	\end{align*}
	Since the Laplacian operator $\Delta$ is invariant, it yields
	\begin{align}\label{2e.p}
	\Delta  q^+ =-\nb_a u^b\nb_b u^a+\mu \nb_a \beta^b \nb_b \beta^a.
	\end{align}
	We have a simple estimate from the assumption \eqref{2energy8} and H\"older's inequality, i.e.,
	\begin{align}\label{deltap2}
	\norm{\Delta q^+ }_{L^2(\Omega^+)}\ls&C\norm{\nb u}_{L^2(\Omega^+)}\norm{\nb u}_{L^\infty(\Omega^+)}+C\norm{\nb \beta}_{L^2(\Omega^+)}\norm{\nb \beta}_{L^\infty(\Omega^+)}\no\\
	\ls& CME_1^{1/2}(t),
	\end{align}
	which is a lower energy term.
	
	For $m\gs 0$, it follows that
	\begin{align*}
	\nb^m\Delta  q^+   =&-\sum_{s=0}^m\bic_m^s \nb^s\nb_a u^b\nb^{m-s}\nb_b u^a+\mu \sum_{s=0}^m\bic_m^s\nb^s\nb_a \beta^b\nb^{m-s}\nb_b \beta^a.
	\end{align*}
	From \eqref{CllemmaA.4},  we get for $s\gs 0$
	\begin{align}\label{r.einfbeta}
	\norm{\nb^s \beta}_{L^\infty(\Omega^+)}\ls &C\sum_{\ell=0}^{2} K_1^{n/2-\ell}\norm{\nb^{\ell+s}\beta}_{L^2(\Omega^+)}
	\ls C(K_1)\sum_{\ell=0}^{2} E_{s+\ell}^{1/2}(t),
	\end{align}
	and, similarly,
	\begin{align}\label{r.einfu}
	\norm{\nb^s u}_{L^\infty(\Omega^+)}\ls C(K_1)\sum_{\ell=0}^{2} E_{s+\ell}^{1/2}(t).
	\end{align}

	From H\"older's inequality, \eqref{r.einfbeta}, Lemma \ref{lem.CLA.3} and \eqref{r.einfu}, we get, 
	\begin{align}\label{deltap34}
	\norm{\nb\Delta q^+}_{L^2(\Omega^+)}\ls &C\norm{\nb u}_{L^\infty(\Omega^+)}\norm{\nb^2u}_{L^2(\Omega^+)}+C\norm{\nb \beta}_{L^\infty(\Omega^+)}\norm{\nb^2\beta}_{L^2(\Omega^+)}\no\\
	\ls & CM E_2^{1/2}(t),\\
	\norm{\nb^2\Delta q^+}_{L^2(\Omega^+)}\ls &C\norm{\nb u}_{L^\infty(\Omega^+)}\norm{\nb^3u}_{L^2(\Omega^+)}+C\norm{\nb^2 u}_{L^4(\Omega^+)}^2\no\\
	&+C\norm{\nb \beta}_{L^\infty(\Omega^+)}\norm{\nb^3\beta}_{L^2(\Omega^+)}+C\norm{\nb^2 \beta}_{L^4(\Omega^+)}^2\no\\
	\ls  &CM E_3^{1/2}(t)+C\norm{\nb u}_{L^\infty(\Omega^+)}\sum_{s=0}^2\norm{\nb^{s+1}u}_{L^2(\Omega^+)}K_1^{2-s}\no\\
	&+C\norm{\nb \beta}_{L^\infty(\Omega^+)}\sum_{s=0}^2\norm{\nb^{s+1}\beta}_{L^2(\Omega^+)}K_1^{2-s}\no\\
	\ls &CM E_3^{1/2}(t)+C(K_1)(E_1^{1/2}(t)+E_2^{1/2}(t)+E_3^{1/2}(t)),\label{deltap35}\\
	\norm{\nb^3\Delta q^+}_{L^2(\Omega^+)}\ls &C\norm{\nb u}_{L^\infty(\Omega^+)}\norm{\nb^4u}_{L^2(\Omega^+)}+C\norm{\nb^2 u}_{L^3(\Omega^+)}\norm{\nb^3u}_{L^6(\Omega^+)}\no\\
	&+C\norm{\nb \beta}_{L^\infty(\Omega^+)}\norm{\nb^4\beta}_{L^2(\Omega^+)}+C\norm{\nb^2 \beta}_{L^3(\Omega^+)}\norm{\nb^3\beta}_{L^6(\Omega^+)}\no\\
	\ls &CM E_4^{1/2}(t)+C(\vol\Omega^+)E_3^{1/2}(t)E_4^{1/2}(t).\label{deltap36}
	\end{align}
	From the definition of the projection and the fact that the measure in the energy is $(-\nb_N P)^{-1}d\mu_\gamma$, we have
	\begin{align*}
	\norm{\Pi\nb^r P}_{L^2(\Gamma)}\ls &  \norm{\nb P}_{L^\infty(\Gamma)}^{1/2}E_r^{1/2}(t).
	\end{align*}
	Thus, by \eqref{CL5.8.1}, \eqref{deltap34}, \eqref{deltap2} and Proposition \ref{prop.vr2}, we obtain for any $2\ls r\ls 4$
	\begin{align}\label{est.nbrp}
	&\norm{\nb^rq^+}_{L^2(\Gamma)}+\norm{\nb^r  q^+ }_{L^2(\Omega^+)}\no\\
	\ls &C\norm{\Pi \nb^r  q^+ }_{L^2(\Gamma)}+C(\tilde{K},\vol\Omega^+)\sum_{s\ls r-1}\norm{\nb^s\Delta  q^+ }_{L^2(\Omega^+)}\no\\
	\ls&C\norm{\nb P}_{L^\infty(\Gamma)}^{1/2}E_r^{1/2}(t)+C(r,K,K_1,\vol\Omega^-) E_0(0)\no\\
	&+C(K,K_1,\vol\Omega^+)[M+(r-2)\sum_{s=1}^{r-1}E_s^{1/2}(t)]E_r^{1/2}(t).
	\end{align}
	Therefore, 
	\begin{align}
	\eqref{r.e12}\ls &C(r,K,K_1,\vol\Omega^-) E_0^2(0)\no\\
	&+\left[C\norm{\nb P}_{L^\infty(\Gamma)}^{1/2}+C(K,K_1,\vol\Omega^+)[M+(r-2)\sum_{s=1}^{r-1}E_s^{1/2}(t)]\right]E_r(t).
	\end{align}

	Now, we turn to the estimates of  \eqref{r.e11}. Due to $P=0$ on $\Gamma$ implying $\gamma_b^a\nb_a P =0$ on $\Gamma$, we have from \eqref{gammaauc}, by noticing $\vartheta=-1/\nb_N P$, that
	\begin{align}\label{nunb}
	-\vartheta^{-1}N_b=&\nb_N P N_b=N^a\nb_a P N_b=\delta_b^a\nb_a P-\gamma_b^a\nb_a P
	=\nb_b P.
	\end{align}
	By   H\"older's inequality and \eqref{nunb}, we get
	\begin{align*}
	|\eqref{r.e11}|\ls C\norm{\vartheta}_{L^\infty(\Gamma)}^{1/2} E_r^{1/2}(t)\norm{\Pi\left(D_t\left(\nb^{r} P\right)+\nb^{r} u\cdot\nb P\right)}_{L^2(\Gamma)}.
	\end{align*}
	It follows from \eqref{commutator} that
	\begin{align}\label{Piterm1}
	D_t\nb^r P+\nb^{r} u\cdot\nb  P 
	=&\sgn(2-r)\sum_{s=1}^{r-2}\bic_{r}^{s+1}(\nb^{s+1}u)\cdot \nb^{r-s} P+\nb^r D_t P.
	\end{align}
	
	We first consider the estimates of the last term in \eqref{Piterm1}. By \eqref{CL5.9.1}, \eqref{CL5.9.3} and \eqref{CLA.7.1}, we have, for $2\ls r\ls 4$ 
	\begin{align}\label{est.pinbdtp}
	&\norm{\Pi\nb^rD_t P}_{L^2(\Gamma)}\no\\
	\ls& 2\norm{\bnb^{r-2}\theta}_{L^2(\Gamma)}\norm{\nb_ND_tP}_{L^\infty(\Gamma)}+C\sum_{k=1}^{r-1} K^k\norm{\nb^{r-k} D_tP}_{L^2(\Gamma)}\no\\
	\ls &C(1/\eps)L\left(\norm{\Pi\nb^r P}_{L^2(\Gamma)}+\sum_{k=1}^{r-1}K^k\norm{\nb^{r-k}P}_{L^2(\Gamma)}\right)\no\\
	&+C\sum_{k=1}^{r-1} K^k\norm{\nb^{r-k} D_tP}_{L^2(\Gamma)}
	\end{align}
	
	By \eqref{CL5.8.2} , it yields
	\begin{align}\label{est.nbkdtp}
	&\norm{\nb^k D_t q^+}_{L^2(\Gamma)}\no\\
	\ls& \delta \norm{\Pi\nb^{k+1} D_t q^+ }_{L^2(\Gamma)}+C(1/\delta, K,\vol\Omega^+)\sum_{s\ls k-1}\norm{\nb^s\Delta D_t  q^+ }_{L^2(\Omega^+)}.
	\end{align}
	
	From \eqref{CL2.4.3}, \eqref{2e.p} and Lemma \ref{lem.CL2.1},  it follows
	\begin{align*}
	\Delta D_t  q^+ 
	=&4g^{ac}\nb_c u^b\nb_a\nb_b q^++(\Delta u^e)\nb_e q^++2\nb_e u^b\nb_b u^a\nb_a u^e\\
	&-2\mu  \nb_b u^a\nb_a\beta^c \nb_c\beta^b-2\mu  \nb_b u^a\beta^c\nb_a\nb_c\beta^b+ 2\mu \nb_b\beta^a\beta^e\nb_e\nb_a u^b.
	\end{align*}
	
	By \eqref{r.einfbeta}, \eqref{est.nbrp} and Lemma \ref{lem.CLA.4}, we get  for $s\ls 2$
	\begin{align*}
	&\norm{\nb^s\Delta D_t q^+}_{L^2(\Omega^+)}\no\\
	\ls&C\norm{\nb  u}_{L^\infty(\Omega^+)}\norm{\nb^{s+2}  q^+ }_{L^2(\Omega^+)}+s(s-1)C\norm{\nb^3 u}_{L^2(\Omega^+)}\norm{\nb^2 q^+ }_{L^\infty(\Omega^+)}\no\\
	&+sC\norm{ \nb^2 u}_{L^4(\Omega^+)}\norm{\nb^{s+1}  q^+ }_{L^4(\Omega^+)}+C\norm{\nb^{s+2} u}_{L^2(\Omega^+)}\norm{\nb   q^+ }_{L^\infty(\Omega^+)}\no\\
	&+C\left(\norm{\nb u}_{L^\infty(\Omega^+)}\norm{\nb u}_{L^\infty(\Omega^+)}+\norm{\nb \beta}_{L^\infty(\Omega^+)}\norm{\nb \beta}_{L^\infty(\Omega^+)}\right) \norm{\nb^{s+1} u}_{L^2(\Omega^+)}\no\\
	&+s(s-1)C\norm{\nb u}_{L^\infty(\Omega^+)}\norm{\nb^2 u}_{L^4(\Omega^+)} \norm{\nb^2 u}_{L^4(\Omega^+)}\no\\
	&+C\norm{\nb u}_{L^\infty(\Omega^+)}\norm{\nb \beta}_{L^\infty(\Omega^+)}\norm{\nb^{s+1} \beta}_{L^2(\Omega^+)}\no\\
	&+sC\norm{\nb^2 u}_{L^4(\Omega^+)}\norm{\nb^2 \beta}_{L^4(\Omega^+)} \left((s-1)\norm{\nb  \beta}_{L^\infty(\Omega^+)}+\norm{\beta}_{L^\infty(\Omega^+)}\right)\no\\
	&+s(s-1)C\norm{\nb u}_{L^\infty(\Omega^+)}\norm{\nb^2 \beta}_{L^4(\Omega^+)} \norm{\nb^2  \beta}_{L^4(\Omega^+)}\no\\
	&+C\norm{\nb u}_{L^\infty(\Omega^+)}\norm{\beta}_{L^\infty(\Omega^+)} \norm{\nb^{s+2} \beta}_{L^2(\Omega^+)}\no\\
	&+sC\norm{\nb^3 u}_{L^2(\Omega^+)}\norm{\beta}_{L^\infty(\Omega^+)} \left((s-1)\norm{\nb^{2} \beta}_{L^\infty(\Omega^+)}+\norm{\nb \beta}_{L^\infty(\Omega^+)}\right)\no\\
	&+s(s-1)C\norm{\nb^3 \beta}_{L^2(\Omega^+)}\norm{\beta}_{L^\infty(\Omega^+)} \norm{\nb^{2} u}_{L^\infty(\Omega^+)}\no\\
	&+s(s-1)C\norm{\nb \beta}_{L^\infty(\Omega^+)}\norm{\nb^2\beta}_{L^4(\Omega^+)} \norm{\nb^{2} u}_{L^4(\Omega^+)}\no\\
	&+s(s-1)C\norm{\nb \beta}_{L^\infty(\Omega^+)}\norm{\beta}_{L^\infty(\Omega^+)} \norm{\nb^4 u}_{L^2(\Omega^+)}\no\\
	&+s(s-1)C\norm{\nb^2 \beta}_{L^\infty(\Omega^+)}\norm{\beta}_{L^\infty(\Omega^+)} \norm{\nb^3 u}_{L^2(\Omega^+)}.
	\end{align*}
	From Lemma~\ref{lem.CLA.3} and \eqref{r.einfu}, it follows for $s\ls 2$
	\begin{align*}
	\norm{\nb^{s+1} u}_{L^4(\Omega^+)}\ls& C\norm{\nb^s u}_{L^\infty(\Omega^+)}^{1/2}\left(\sum_{\ell=0}^2\norm{\nb^{s+\ell} u}_{L^2(\Omega^+)}K_1^{2-\ell}\right)^{1/2}\no\\
	\ls& C(K_1)\sum_{\ell=0}^2 E_{s+\ell}^{1/2}(t).
	\end{align*}
	All the terms with $L^4(\Omega^+)$ norms can be estimated in the same way with the help of \eqref{r.einfbeta}, \eqref{r.einfu}, the similar estimates of $q^+$ and the assumptions. Then, we obtain the bound which is linear about the highest-order derivative or the highest-order energy $E_r^{1/2}(t)$, i.e.,
	\begin{align}\label{nbdtp}
	\norm{\nb^s\Delta D_t q^+ }_{L^2(\Omega^+)}\ls& C(K,K_1,M,M_1,L,1/\eps,\vol\Omega^+,E_0(0))\no\\
	&\qquad\cdot\Big(1+\sum_{\ell=0}^{r-1}E_\ell(t)\Big)\big(1+E_r^{1/2}(t)\big).
	\end{align}
	
	Because of
	\begin{align*}
	D_t q^-=\frac{\mu}{2} D_t|\varpi|^2=-\mu \varpi^a(\nb\times\varXi)_a+\mu u^c\varpi^a\nb_c\varpi_a,
	\end{align*}
	it follows from Propositions~\ref{prop.v.r} and \ref{prop.E.r} that
	\begin{align*}
	\norm{\nb^kD_tq^-}_{L^2(\Gamma)}\ls& \mu C(K)E_0(0)\sum_{s=0}^{k+1}E_s^{1/2}(t)\\
	&+\mu C(r,K,\vol\Omega^+) E_0(0)\norm{\nb\times\varXi_0}_{L^2(\Omega^-)} e^{C(K,L)t}.
	\end{align*}
	Then, from \eqref{est.nbrp}, \eqref{est.pinbdtp}, \eqref{est.nbkdtp}, \eqref{nbdtp} and taking some small $\delta$'s which are independent of $E_r(t)$, we obtain, by the induction argument for $r$, that
	\begin{align}\label{pinbrdtp}
	\norm{\Pi \nb^r D_t q^+ }_{L^2(\Gamma)}\ls &\mu C(r,K) E_0(0)\norm{\nb\times\varXi_0}_{L^2(\Omega^-)} e^{C(K,L)t}\no\\
	&+(1+\mu)C(K,K_1,M,M_1,L,1/\eps,\vol\Omega,E_0(0))\no\\
	&\qquad\cdot\Big(1+\sum_{\ell=0}^{r-1}E_\ell(t)\Big)\big(1+E_r^{1/2}(t)\big).
	\end{align}
	Then, we have the similar bound for $\norm{\Pi \nb^r D_t P }_{L^2(\Gamma)}$.
	
	For \eqref{Piterm1}, it only remains to estimate
	\begin{align*}
	\norm{\Pi\left((\nb^{s+1}u)\cdot \nb^{r-s} P\right)}_{L^2(\Gamma)} \text{ for } 1\ls s\ls r-2.
	\end{align*}
	For the cases $r=3,4$ and $s=r-2$, we get, from \eqref{2energy92} and Lemma \ref{lem.CLA.7}, that
	\begin{align*}
	&\norm{\Pi\left(\nb^{r-1}u\cdot \nb^2 P\right)}_{L^2(\Gamma)}
	\ls \norm{\nb^{r-1} u}_{L^2(\Gamma)}\norm{\nb^2  P }_{L^\infty(\Gamma)}\no\\
	\ls &C(K,\vol\Omega^+)L(\vol\Omega^+)^{1/6}\left(\norm{\nb^r u}_{L^2(\Omega)}+\norm{\nb^{r-1} u}_{L^2(\Omega)}\right)\no\\
	\ls &C(K,L,\vol\Omega^+)\left(E_{r-1}^{1/2}(t)+E_r^{1/2}(t)\right).
	\end{align*}
	For the case $r=4$ and $s=1$, by \eqref{CL4.48}, Lemma \ref{lem.CLA.7}, \eqref{est.nbrp} and Proposition \ref{prop.v.r}, we have
	\begin{align*}
	&\norm{\Pi\left(\nb^2u\cdot \nb^3 P\right)}_{L^2(\Gamma)}
	=\norm{\Pi\nb^2u\cdot \Pi\nb^3 P+\Pi(\nb^2 u\cdot N)\tilde{\otimes}\Pi(N\cdot \nb^3P)}_{L^2(\Gamma)}\no\\
	\ls&C\norm{\Pi\nb^2 u}_{L^4(\Gamma)}\norm{\Pi\nb^3 P}_{L^4(\Gamma)}+C\norm{\Pi(N^a\nb^2 u_a)}_{L^4(\Gamma)}\norm{\Pi(\nb_N\nb^2 P)}_{L^4(\Gamma)}\no\\
	\ls &C\norm{\nb^2 u}_{L^4(\Gamma)}\norm{\nb^3  P}_{L^4(\Gamma)}\no\\
	\ls&C(K,\vol\Omega^+)\left(\norm{\nb^3 u}_{L^2(\Omega^+)}+\norm{\nb^2 u}_{L^2(\Omega^+)}\right)\no\\
	&\cdot \left(\norm{\nb^4  q^+ }_{L^2(\Omega^+)}+\norm{\nb^3  q^+ }_{L^2(\Omega^+)}+\norm{\nb^3  q^-}_{L^4(\Gamma)}\right)\no\\
	\ls&C(K, K_1,\vol\Omega^+)(E_3^{1/2}(t)+E_2^{1/2}(t))\no\\
	&\cdot\left(\sum_{s=0}^3 E_s(t)+\left(\sum_{\ell=0}^{2}E_\ell^{1/2}(t)\right)E_4^{1/2}(t)+C(K,\vol\Omega^-)E_0(0)\right)\no\\
	\ls &C(K, K_1,\vol\Omega)\sum_{s=0}^3 E_s(t)\sum_{\ell=0}^4 E_\ell^{1/2}(t).
	\end{align*}
	
	Thus, we get
	\begin{align*}
	|\eqref{r.e11}|\ls& C(K,K_1,M,M_1,L,1/\eps,\vol\Omega,E_0(0)) \Big(1+\sum_{s=0}^{r-1}E_s(t)\Big)\big(1+E_r(t)\big)\\
	&+\mu C(r,K,\vol\Omega^+) E_0(0)\norm{\nb\times\varXi_0}_{L^2(\Omega^-)} e^{C(K,L)t}E_r^{1/2}(t).
	\end{align*}
	
	Therefore, we have obtained
	\begin{align*}
	&|\eqref{r.e1}+\eqref{r.e2}+\eqref{r.e7}|\\
	\ls &C(K,K_1,M,M_1,L,1/\eps,\vol\Omega,E_0(0)) \Big(1+\sum_{s=0}^{r-1}E_s(t)\Big)\big(1+E_r(t)\big)\\
	&+\mu C(r,K) E_0(0)\norm{\nb\times\varXi_0}_{L^2(\Omega^-)} e^{C(K,L)t}E_r^{1/2}(t).
	\end{align*}
	
	By a similar argument in \cite[(5.65)]{HLarma}, we get
	\begin{align*}
	|\eqref{r.e3}|\ls C(K,K_1,M,\vol\Omega,1/\eps)\left(1+\sum_{s=0}^{r-1} E_s(t)\right)E_r(t).
	\end{align*}
	
	From \eqref{CL3.9.1} and \eqref{Dtpcommu}, we have
	\begin{align*}
	D_t(\nb_N P)=-2h_d^a N^d\nb_a P+h_{NN}\nb_N P+\nb_N D_tP,
	\end{align*}
	which yields
	\begin{align*}
	\frac{\vartheta_t}{\vartheta}=-\frac{D_t\nb_N P }{\nb_N P}=\frac{2h_d^a N^d\nb_a  P}{\nb_N P}-h_{NN}-\frac{\nb_ND_t P}{\nb_N P}=h_{NN}-\frac{\nb_ND_t P}{\nb_N P}.
	\end{align*}
	Thus, we can easily obtain that \eqref{r.e8} can be controlled by $C(K,M,L,1/\eps)E_r(t)$.
	
	Note that there always exists a constant $C>0$ such that $\norm{\nb\times\varXi_0}_{L^2(\Omega^-)}^2\ls CE_0(0)$ at initial time. Therefore, we obtain
	\begin{align*}
	\frac{d}{dt}E_r(t)\ls &C(K,K_1,M,M_1,L,1/\eps,\vol\Omega,E_0(0)) \Big(1+\sum_{s=0}^{r-1}E_s(t)\Big)\big(1+E_r(t)\big)\\
	&+\sgn((r-1)(r-2))\mu^2 C(r,K) E_0(0) e^{C(K,L)t},
	\end{align*}
	which implies, by Gr\"onwall's inequality, that
	\begin{align*}
	E_r(t)\ls & E_r(0)\exp\left(C(K,K_1,M,M_1,L,1/\eps,\vol\Omega,E_0(0)) \int_0^t\Big(1+\sum_{s=0}^{r-1}E_s(\tau)\Big)d\tau\right)\\
	&+\Big\{ C(K,K_1,M,M_1,L,1/\eps,\vol\Omega,E_0(0)) \\
	&\qquad+\sgn((r-1)(r-2))\mu^2 e^{C(K,L)t}\Big\}\\
	&\cdot\int_0^t\Big(1+\sum_{s=0}^{r-1}E_s(\tau)\Big)\exp\Bigg(C(K,K_1,M,M_1,L,1/\eps,\vol\Omega,E_0(0))\\
	&\qquad\qquad\qquad\qquad\qquad\qquad\cdot \int_\tau^t\Big(1+\sum_{s=0}^{r-1}E_s(s)\Big)ds\Bigg)d\tau.
	\end{align*}
	By using induction for $r=1,2,3,4$ in turn, we obtain the desired estimates.
\end{proof}

\section{Justification of A Priori Assumptions}\label{sec.6}

Let  $\K(t)$ and $\eps(t)$ be the maximum and minimum values, respectively, such that \eqref{2energy9} and \eqref{2energy91} hold at time $t$:
\begin{align}
\K(t)=&\max\left(\norm{\theta(t,\cdot)}_{L^\infty(\Gamma)}, 1/\iota_0(t)\right),\label{K}\\
\E(t)=&\norm{1/(\nb_N \p(t,\cdot))}_{L^\infty(\Gamma)}=1/\eps(t). \label{E}
\end{align}

\begin{lemma}\label{lem.7.6}
	Assume $\norm{\theta}_{L^\infty(W)}+1/\iota_0'\ls C\K(0)$. Let $K_1\gs 1/\iota_1$ be as in Definition \ref{defn.3.5}, $\E(t)$ as in \eqref{E}. Then there are continuous functions $G_j$, $j=1,2,3,4,5$, such that
	\begin{align}
	\norm{\nb u}_{L^\infty(\Omega)}+\norm{\nb \beta}_{L^\infty(\Omega)}+&\norm{ \beta}_{L^\infty(\Omega)}\ls G_1(K_1,E_0,\cdots, E_{4}),\label{e.1}\\
	\norm{\nb P}_{L^\infty(\Omega)}+\norm{\nb^2 P}_{L^\infty(\Gamma)}\ls &G_2(K_1,\E,E_0,\cdots, E_{4},\vol\Omega),\label{e.2}\\
	\norm{\theta}_{L^\infty(\Gamma)}\ls &G_3(K_1,\E,E_0,\cdots, E_{4},\vol\Omega),\label{e.3}\\
	\norm{u}_{L^\infty(\Gamma)}+\norm{\nb D_t P}_{L^\infty(\Gamma)}\ls &G_4(K_1,\E,E_0,\cdots, E_{4},\vol\Omega),\label{e.4}\\
	\norm{u}_{L^\infty(\Gamma)}+\sum_{\ell=0}^{2}\norm{\nb u }_{L^2(\Gamma)}\ls &G_5(K_1,\E,E_0,\cdots, E_{4},\vol\Omega).\label{e.4'}
	\end{align}
\end{lemma}

\begin{proof}
	\eqref{e.1} follows from \eqref{r.einfu}, \eqref{r.einfbeta} and \eqref{CllemmaA.4}. \eqref{e.2} follows from Lemmas \ref{lem.CLA.4} and \ref{lem.CLA.2}, Lemmas \ref{lem.CLA.5}--\ref{lem.CLA.7}, and \eqref{deltap2}. Since, from \eqref{CL4.20},
	\begin{align}\label{e.7}
	|\nb^2P|\gs |\Pi\nb^2P|=|\nb_NP||\theta|\gs \E^{-1}|\theta|,
	\end{align}
	\eqref{e.3} follows from \eqref{e.2}. \eqref{e.4} follows from Lemma \ref{lem.CLA.2}, \eqref{est.nbkdtp}, \eqref{nbdtp} and \eqref{pinbrdtp}. \eqref{e.4'} follows from Lemmas \ref{lem.CLA.2} and \ref{lem.CLA.7}.
\end{proof}

\begin{lemma}\label{lem.7.7}
	Assume $\norm{\theta}_{L^\infty(W)}+1/\iota_0'\ls C\K(0)$. Let $K_1\gs 1/\iota_1$ and $\eps_1$ be as in Definition \ref{defn.3.5}. Then
	\begin{align}\label{e.8}
	\abs{\frac{d}{dt}E_r}\ls C_r(K_1,\E,E_0,\cdots, E_{4},\vol\Omega)\sum_{s=0}^r E_s,
	\end{align}
	and
	\begin{align}\label{e.9}
	\abs{\frac{d}{dt}\E}\ls C_r(K_1,\E,E_0,\cdots, E_{4},\vol\Omega).
	\end{align}
\end{lemma}

\begin{proof}
	\eqref{e.8} is a consequence of Lemma \ref{lem.7.6} and the estimates in the proof of Theorems  \ref{thm.renergy}. \eqref{e.9} follows from \eqref{e.4} and 
	\begin{align*}
	\abs{\frac{d}{dt}\lnorm{\frac{1}{-\nb_NP(t,\cdot)}}_{L^\infty(\Gamma)}}\ls C\lnorm{\frac{1}{-\nb_NP(t,\cdot)}}_{L^\infty(\Gamma)}^2 \norm{\nb_ND_tP(t,\cdot)}_{L^\infty(\Gamma)}.
	\end{align*}
\end{proof}

As a consequence of Lemma \ref{lem.7.7}, we have the following:

\begin{lemma}\label{lem.7.8}
	Assume $\norm{\theta}_{L^\infty(W)}+1/\iota_0'\ls C\K(0)$.  There exists a continuous function $\T>0$ depending on $K_1$, $\E(0)$, $E_0(0)$, $\cdots$, $E_{n+1}(0)$, $\vol\Omega$ such that for
	\begin{align*}
	0\ls t\ls \T(K_1,\E(0),E_0(0),\cdots, E_{4}(0),\vol\Omega),
	\end{align*}
	the following statements hold: We have
	\begin{align}\label{7.37}
	E_s(t)\ls 2E_s(0), \quad 0\ls s\ls 4; \quad \E(t)\ls 2\E(0).
	\end{align}
	Furthermore,
	\begin{align}\label{7.38}
	\frac{g_{ab}(0,y)Y^aY^b}{2}\ls g_{ab}(t,y)Y^aY^b\ls 2g_{ab}(0,y)Y^aY^b,
	\end{align}
	and with $\eps_1$ as in Definition \ref{defn.3.5},
	\begin{align}
	\qquad\abs{\N(x(t,\bar{y}))-\N(x(0,\bar{y}))}\ls&\frac{\eps_1}{16}, &&\bar{y}\in\Gamma,\qquad\label{7.39}\\
	\abs{x(t,y)-x(t,y)}\ls&\frac{\iota_1}{16}, &&y\in\Omega,\label{7.40}\\
	\abs{\frac{\D x(t,\bar{y})}{\D y}-\frac{\D (0,\bar{y})}{\D y}}\ls &\frac{\eps_1}{16}, &&\bar{y}\in\Gamma.\label{7.41}
	\end{align}
\end{lemma}

\begin{proof} Since the proof is similar to \cite[Lemma 6.3]{HLarma}, we omit the details.
\end{proof}

Now we use \eqref{7.38}--\eqref{7.41} to pick a $K_1$, i.e., $\iota_1$, which depends only on its value at $t=0$,
\begin{align*}
\iota_1(t)\gs \iota_1(0)/2.
\end{align*}

\begin{lemma}\label{lem.7.9}
	Assume $\norm{\theta}_{L^\infty(W)}+1/\iota_0'\ls C\K(0)$. Let $\T$ be as in Lemma \ref{lem.7.7}. Pick $\iota_1>0$ such that
	\begin{align}\label{7.55}
	\abs{\N(x(0,y_1))-\N(x(0,y_2))}\ls \frac{\eps_1}{2}, \text{ whenever } \abs{x(0,y_1)-x(0,y_2)}\ls 2\iota_1.
	\end{align}
	Then if $t\ls \T$, we have
	\begin{align}\label{7.56}
	\abs{\N(x(t,y_1))-\N(x(t,y_2))}\ls \eps_1, \text{ whenever } \abs{x(t,y_1)-x(t,y_2)}\ls 2\iota_1.
	\end{align}
\end{lemma}

\begin{proof}
	\eqref{7.56} follows from \eqref{7.55}, \eqref{7.39} and \eqref{7.40} in view of triangle inequalities.
\end{proof}

Finally, Lemma \ref{lem.7.9} allows us to pick a $K_1$ depending only on initial conditions, while Lemma \ref{lem.7.8} gives us $\T>0$ that depends only on the initial conditions and $K_1$ such that, by Lemma \ref{lem.7.9}, $1/\iota_1\ls K_1$ for $t\ls \T$. Thus, we immediately obtain Theorem \ref{main thm}.

\appendix

\section{Preliminaries and Some Estimates}

Let us now recall some properties of the projection. Since $g^{ab}=\gamma^{ab}+N^aN^b$, we have
\begin{align}\label{CL4.48}
\Pi(S\cdot R)=\Pi(S)\cdot \Pi(R)+\Pi(S\cdot N)\tilde{\otimes}\Pi(N\cdot R),
\end{align}
where $S\tilde{\otimes} R$ denotes some partial symmetrization of the tensor product $S\otimes R$, i.e., a sum over some subset of the permutations of the indices divided by the number of permutations in that subset. Similarly, we let $S\tilde{\cdot} R$ denote a partial symmetrization of the dot product $S\cdot R$. Now we recall some identities:
\begin{align}
\Pi\nb^2 q=&\bnb^2 q+\theta \nb_N q,\label{CL4.20}\\
\Pi\nb^3 q=&\bnb^3 q-2\theta\tilde{\otimes}(\theta\tilde{\cdot}\bnb q)+(\bnb\theta)\nb_N q+3\theta\tilde{\otimes}\bnb\nb_N q,\label{CL4.21}\\
\Pi\nb^4 q=&\bnb^4 q-\theta\tilde{\otimes}\left(5(\bnb\theta)\tilde{\cdot}\bnb q+8\theta\tilde{\cdot}\bnb^2 q\right)-2(\bnb\theta)\tilde{\otimes}(\theta\tilde{\cdot}\bnb q)+(\bnb^2\theta)\nb_Nq\no\\
&+4(\bnb\theta)\tilde{\otimes}\bnb\nb_Nq+6\theta\tilde{\otimes}\bnb^2\nb_N q-3\theta\tilde{\otimes}(\theta\tilde{\cdot}\theta)\nb_N q+3\theta\tilde{\otimes}\theta \nb_N^2q.\label{CL4.22}
\end{align}

\begin{definition}\label{defn.3.3}
	Let $\N(\bar{x})$ be the outward unit normal to $\Gamma_t$ at $\bar{x}\in \Gamma_t$. Let $\dist(x_1,x_2)=|x_1-x_2|$ denote the Euclidean distance in $\R^n$, and for $\bar{x}_1, \bar{x}_2\in \Gamma_t$, let $\dist_{\Gamma_t} (\bar{x}_1, \bar{x}_2)$ denote the geodesic distance on the boundary.
\end{definition}

\begin{definition}\label{defn.3.5}
	Let $0<\eps_1<2$ be a fixed number, and let $\iota_1=\iota_1(\eps_1)$ the largest number such that
	\begin{align*}
	\abs{\N(\bar{x}_1)-\N(\bar{x}_2)}\ls \eps_1 \quad \text{whenever } \abs{\bar{x}_1-\bar{x}_2}\ls \iota_1, \; \bar{x}_1,\bar{x}_2\in\Gamma_t.
	\end{align*}
\end{definition}

\begin{lemma}[cf. \mbox{\cite[Lemma 3.9]{CL}}] \label{lem.CL3.9}
	Let $N$ be the unit normal to $\Gamma$, and let $h_{ab}=\frac{1}{2}D_tg_{ab}$. On $[0,T]\times \Gamma$, we have
	\begin{align}\label{CL3.9.1}
	&D_tN_a=h_{NN}N_a,\; D_tN^c=-2h_d^cN^d+h_{NN}N^c,\; D_t\gamma^{ab}=-2\gamma^{ac}h_{cd}\gamma^{db},
	\end{align}
	where $h_{NN}=h_{ab}N^aN^b$. The volume element on $\Gamma$ satisfies
	\begin{align}\label{CL3.9.3}
	D_td\mu_\gamma=(\tr h-h_{NN})d\mu_\gamma.
	\end{align}
\end{lemma}

\begin{lemma}[\mbox{cf. \cite[Lemma 5.5]{CL}}] \label{lem.CL5.5}
	Let $w_a=w_{Aa}=\nb_A^r f_a$, $\nb_A^r=\nb_{a_1}\cdots \nb_{a_r}$, $f$ be a $(0,1)$ tensor, and $[\nb_a,\nb_b]=0$. Let $\dv w=\nb_a w^a=\nb^r\dv f$, and let $(\curl w)_{ab}=\nb_aw_b-\nb_b w_a=\nb^r(\curl f)_{ab}$. Then,
	\begin{align}\label{divcurl}
	|\nb w|^2\ls& C(g^{ab}\gamma^{cd}\gamma^{AB}\nb_c w_{Aa}\nb_d w_{Bb}+|\dv w|^2+|\curl w|^2),\\
	\int_{\Omega}|\nb w|^2 d\mu_g\ls& C\int_\Omega (N^aN^fg^{cd}\gamma^{AF}\nb_c w_{Aa}\nb_d w_{Ff}+|\dv w|^2+|\curl w|^2+K^2|w|^2)d\mu_g.\label{divcurlint}
	\end{align}
\end{lemma}

\begin{lemma}[\mbox{cf. \cite[Proposition 5.8]{CL}}] \label{lem.CL5.8}
	Let $\iota_0$ and $\iota_1$ be as in \eqref{defn.3.4} and Definition  \ref{defn.3.5}, and suppose that $|\theta|+1/\iota_0\ls K$ and $1/\iota_1\ls K_1$. Then with $\tilde{K}=\min(K,K_1)$ we have, for any $r\gs 2$ and $\delta>0$,
	\begin{align}
	&\norm{\nb^r q}_{L^2(\Gamma)}+\norm{\nb^r q}_{L^2(\Omega)}\no\\
	&\qquad\ls C\norm{\Pi \nb^r q}_{L^2(\Gamma)}+C(\tilde{K},\vol\Omega)\sum_{s\ls r-1} \norm{\nb^s\Delta q}_{L^2(\Omega)},\label{CL5.8.1}\\
	&\norm{\nb^{r-1} q}_{L^2(\Gamma)}+\norm{\nb^r q}_{L^2(\Omega)}\no\\
	&\qquad\ls \delta\norm{\Pi \nb^r q}_{L^2(\Gamma)}+C(1/\delta,K,\vol\Omega)\sum_{s\ls r-2} \norm{\nb^s\Delta q}_{L^2(\Omega)}.\label{CL5.8.2}
	\end{align}
\end{lemma}

\begin{lemma}[cf. \mbox{\cite[Proposition 5.9]{CL}}] \label{lem.CL5.9}
	Assume that $2\ls r\ls 4$. Suppose that $|\theta|\ls K$ and $\iota_1\gs 1/K_1$, where $\iota_1$ is as in Definition 3.5 of \cite{CL}. If $q=0$ on $\Gamma$, then 
	\begin{align}\label{CL5.9.1}
	\norm{\Pi\nb^r q}_{L^2(\Gamma)}\ls & 2\norm{\bnb^{r-2}\theta}_{L^2(\Gamma)}\norm{\nb_Nq}_{L^\infty(\Gamma)}+C\sum_{k=1}^{r-1} K^k\norm{\nb^{r-k} q}_{L^2(\Gamma)}.
	\end{align}
	If, in addition, $|\nb_N q|\gs \eps>0$ and $|\nb_Nq|\gs 2\eps\norm{\nb_Nq}_{L^\infty(\Gamma)}$, then
	\begin{align}\label{CL5.9.3}
	\norm{\bnb^{r-2}\theta}_{L^2(\Gamma)}\ls C(1/\eps)\left(\norm{\Pi\nb^r q}_{L^2(\Gamma)}+\sum_{k=1}^{r-1}K^k\norm{\nb^{r-k}q}_{L^2(\Gamma)}\right).
	\end{align}
\end{lemma}

\begin{lemma}[\mbox{\cite[Lemma A.1]{CL}}] \label{lem.CLA.1}
	If $\alpha$ is a $(0,r)$ tensor, then with $a=k/m$ and a constant $C$ that only depends on $m$ and $n$, such that
	\begin{align*}
	\norm{\bnb^k\alpha}_{L^s(\Gamma)}\ls C\norm{\alpha}_{L^q(\Gamma)}^{1-a}\norm{\bnb^m \alpha}_{L^p(\Gamma)}^a,
	\end{align*}
	if
	\begin{align*}
	\frac{m}{s}=\frac{k}{p}+\frac{m-k}{q}, \quad 2\ls p\ls s\ls q\ls \infty.
	\end{align*}
\end{lemma}

\begin{lemma}[\mbox{\cite[Lemma A.2]{CL}}] \label{lem.CLA.2}
	Suppose that for $\iota_1\gs 1/K_1$
	\begin{align*}
	\abs{\N(\bar{x}_1)-\N(\bar{x}_2)}\ls \eps_1, \quad \text{whenever } |\bar{x}_1-\bar{x}_2|\ls \iota_1, \; \bar{x}_1,\bar{x}_2\in\Gamma_t,
	\end{align*}
	and
	\begin{align*}
	C_0^{-1}\gamma_{ab}^0(y) Z^aZ^b\ls \gamma_{ab}(t,y)Z^aZ^b\ls C_0\gamma_{ab}^0(y) Z^aZ^b, \quad \text{if } Z\in T(\Omega^+),
	\end{align*}
	where $\gamma_{ab}^0(y)=\gamma_{ab}(0,y)$. Then if $\alpha$ is a $(0,r)$ tensor,
	\begin{align}
	&\norm{\alpha}_{L^{(n-1)p/(n-1-kp)}(\Gamma)}\ls C(K_1) \sum_{\ell=0}^k \norm{\nb^\ell \alpha}_{L^p(\Gamma)}, \quad 1\ls p<\frac{n-1}{k},\\
	&\norm{\alpha}_{L^\infty(\Gamma)}\ls \delta\norm{\nb^k \alpha}_{L^p(\Gamma)}+C_\delta(K_1)\sum_{\ell=0}^{k-1} \norm{\nb^\ell \alpha}_{L^p(\Gamma)}, \quad k>\frac{n-1}{p},\label{A.32}
	\end{align}
	for any $\delta>0$.
\end{lemma}

\begin{lemma}[\mbox{\cite[Lemma A.3]{CL}}] \label{lem.CLA.3}
	With notation as in Lemmas \ref{lem.CLA.1} and \ref{lem.CLA.2}, we have
	\begin{align*}
	\sum_{j=0}^k\norm{\nb^j\alpha}_{L^s(\Omega)}\ls C\norm{\alpha}_{L^q(\Omega)}^{1-a}\left(\sum_{i=0}^m \norm{\nb^i\alpha}_{L^p(\Omega)}K_1^{m-i}\right)^a.
	\end{align*}
\end{lemma}

\begin{lemma}[\mbox{\cite[Lemma A.4]{CL}}] \label{lem.CLA.4}
	Suppose that $\iota_1\gs 1/K_1$ and $\alpha$ is a $(0,r)$ tensor. Then
	\begin{align}
	\norm{\alpha}_{L^{np/(n-kp)}(\Omega)} \ls& C\sum_{\ell=0}^k K_1^{k-\ell} \norm{\nb^\ell \alpha}_{L^p(\Omega)}, \quad 1\ls p<\frac{n}{k},\label{A.4.1}\\
	\norm{\alpha}_{L^\infty(\Omega)}\ls &C\sum_{\ell=0}^k K_1^{n/p-\ell} \norm{\nb^\ell \alpha}_{L^p(\Omega)}, \quad k>\frac{n}{p}.\label{A.4.2}
	\end{align}
\end{lemma}

\begin{lemma}[\mbox{\cite[Lemma A.5]{CL}}] \label{lem.CLA.5}
	Suppose that $q=0$ on $\Gamma$. Then
	\begin{align*}
	\norm{q}_{L^2(\Omega)}\ls&C(\vol\Omega)^{1/n}\norm{\nb q}_{L^2(\Omega)},\;
	\norm{\nb q}_{L^2(\Omega)}\ls C(\vol\Omega)^{1/2n}\norm{\Delta q}_{L^2(\Omega)}.
	\end{align*}
\end{lemma}

\begin{lemma}[\mbox{\cite[Lemma A.7]{CL}}] \label{lem.CLA.7}
	Let $\alpha$ be a $(0,r)$ tensor. Assume that
	$$\vol\Omega \ls V \text{ and  }\norm{\theta}_{L^\infty(\Gamma)}+1/\iota_0 \ls K,$$
	then there is a $C=C(K,V,r,n)$ such that
	\begin{align}
	&\norm{\alpha}_{L^{(n-1)p/(n-p)}(\Gamma)} \ls C\norm{\nb \alpha}_{L^p(\Omega)} +C\norm{\alpha}_{L^p(\Omega)},\quad 1\ls p<n,\label{CLA.7.1}\\
	&\norm{\nb^2\alpha}_{L^2(\Omega)} \ls C\left(\norm{\Pi\nb^2\alpha}_{L^{2(n-1)/n}(\Gamma)} +\norm{\Delta\alpha}_{L^2(\Omega)}+\norm{\nb\alpha}_{L^2(\Omega)}\right). \label{CLA.7.2}
	\end{align}
\end{lemma}




\end{document}